\documentclass[a4paper, 11pt]{article}
\usepackage{amsmath}
\usepackage{amssymb,esint}
\usepackage{amscd}
\usepackage{xspace}
\usepackage{fancyhdr}
\usepackage{color}
\usepackage{srcltx} 

\newtheorem{theorem}{Theorem}[section]

\newtheorem{lemma}[theorem]{Lemma}

\newtheorem{remark}[theorem]{Remark}

\newenvironment{proof}[1][Proof]{\textbf{#1.} }{\hfill\rule{0.5em}{0.5em}}
{\catcode`\@=11\global\let\AddToReset=\@addtoreset
\AddToReset{equation}{section}

\AddToReset{theorem}{section}

\newcommand{\cM}{{\mathcal M}}
\newcommand{\cP}{{\mathcal P}}

\newcommand{\cU}{{\mathcal U}}

\newcommand{\cG}{{\mathcal G}}

\newcommand{\cS}{{\mathcal S}}

\newcommand{\A}{{\mathbb A}}
\newcommand{\R}{{\mathbb R}}

\newcommand{\N}{{\mathbb N}}
\newcommand{\F}{{\mathbb F}}

\newcommand{\Pn}{{\mathcal P}}

\newcommand{\supp}{{\rm supp}}

\newcommand{\argmin}{{\rm argmin}}
     
\newcommand{\CE}{{\rm CE}}

\newcommand{\dive}{{\rm div}}    
\newcommand{\W}{{\rm W}}    
\newcommand{\vv}{{\bf v}}    
\newcommand{\mdiv}{\operatorname{div}}

\begin{document}
\title{Gradient flows of modified Wasserstein distances and porous medium equations with nonlocal pressure}
	\author{
	{Nhan-Phu Chung\thanks{E-mail address: phuchung@skku.edu, phuchung82@gmail.com, Department of Mathematics, Sungkyunkwan University, 2066 Seobu-ro, Jangan-gu, Suwon-si, Gyeonggi-do, Korea 16419.} and Quoc-Hung Nguyen\thanks{E-mail address: qhnguyen@amss.ac.cn, Academy of Mathematics and Systems Science,
Chinese Academy of Sciences,
Beijing 100190, PR China.} }}

\maketitle
\begin{center}
Dedicated to Professor Duong Minh Duc on the occasion of his 70th birthday.
\end{center}
\begin{abstract}
We study families of porous medium equation with nonlocal pressure. We construct their weak solutions via JKO schemes for modified Wasserstein distances. We also establish the regularization effect and decay estimates for the $L^p$ norms.

\medskip

\medskip

\medskip

\noindent 

\medskip

\noindent
\end{abstract}   
                  
									
 \section{Introduction and main results} 

 Discretizing a time dependent PDE as a series of variational problem has been a common practice for a while. For example, consider the diffusion equation
 $$\partial_tu(t,x) -\Delta u(t,x) = 0, \quad x \in \R^d.$$ 
The classical discretization scheme reads as
 $$u^{k} = \argmin_u F(u),~\text{with}~F(u):= \left\{\frac{1}{\tau}\|u - u^{k-1}\|^2_{L^2(\R^d)} + \|\nabla u\|^2_{L^2(\R^d)} \right\}.$$
 This comes from the idea that diffusion equation is the gradient flux of the Dirichlet integral $\|\nabla u\|^2_{L^2(\R^d)}$. It also can be seen by considering the additive perturbation $u^k +t\varphi$ of the solution $u^k$, where $\varphi \in C^\infty_0(\R^d)$. Then, the equation $\left.\frac{d}{dt}F(u+t\varphi)\right| _{|t=0}=0$ implies 
 $$\int_{\R^d} [(u^k-u^{k-1}) \varphi + \tau \nabla u^k \nabla \varphi] dx=0.$$
That is, $u^k$ satisfies the backward finite difference scheme for the diffusion equation $u^k = u^{k-1} + \tau \Delta u^k$.
\vspace{1cm}\\A novel discretization scheme, called JKO scheme, was proposed by Jordan–Kinderlehrer–Otto in the seminal paper  \cite{JKO}, for $\tau>0$, 
\begin{equation}\label{E:wva}
	u^k = \argmin_{u\in \Pn_2(\R^d)} F_w(u),~\text{with}~F_w(u):=\left\{\frac{1}{2 \tau} \mathbf{W}^2_2(u,u^{k-1}) + \int_{\R^d} u \, \log u \, dx \right\},
\end{equation}
where $\mathbf{W}_2$ is the Wasserstein distance and $\Pn_2(\R^d)$ is the space of probability Borel measures on $\R^d$ with finite second order moment (see section 3).

The main idea in proving (\ref{E:wva}) is the perturbation by domain deformations \cite{JKO}. Namely, given $\eta\in C^\infty_c(\mathbb{R}^d)$, let $\Phi_t(x)=x+t \eta(x)$ and $u_t$ be the push forward of $u^{k}$ under $\Phi_t$. That is, for all $\varphi \in C_0^\infty(\R^d)$,
%

 \begin{equation*} \label{E:pf}
 \int_{\R^d} \varphi(x) \, u_t (x)\, dx=\int_{\R^d} \varphi(\Phi_t(x))\, u^{k}(x) \, dx.
 \end{equation*}
Then, one can obtain the a few key formulas
\begin{equation} \label{E:deri} \frac{du_t}{dt} = - \nabla(\eta u),\end{equation}
and, for $\eta = \nabla \phi$
 \begin{equation} \label{E:Wd} \frac{d}{dt}\left[\frac{1}{2\tau}\mathbf{W}^2_2(u_t,u^{k})\right]_{t=0} = \int_{\R^d} \frac{u^{k}- u^{k-1}}{\tau} \, \phi \,dx + \Lambda(\tau),\end{equation} where $\Lambda(\tau) \to 0$ as $\tau \to 0$, in some proper sense, and
 $$\frac{d}{dt} \int_{\R^d} u_t \log u_t \, dx |_{t=0}=- \int_{\R^d} \Delta u^{k} \phi dx,$$
 see \cite{JKO} for more details. 
The equation $\frac{d F_w(u_t)}{dt} =0$ (coming from the variational formulation (\ref{E:wva})), then heuristically gives the backward finite difference approximation
$$u^{k} \approx u^{k-1} + \tau \Delta u^{k}.$$
 
Moreover, let us define $u_\tau(t) = u^k$ for $t \in [k\tau, (k+1) \tau)$. Then, $u_\tau(t) \to u(t)$ weakly in in $L^1(\R^d)$ for almost all $t \in \R$ (see \cite{JKO}), where $u$ satisfies the variational form of the diffusion equation
$$\int_0^\infty \int_{\R^d} u (\phi_t - \Delta \phi) dx dt = \int_{\R^d} u(0)\phi(0) dx.$$
Therefore, the JKO scheme is a proper discretization for the diffusion equation.

Let us review a few extensions of the JKO scheme. First, we consider the equation
\begin{equation}\label{E:main}
\partial_t u - \operatorname{div}(u\nabla (-\Delta)^{-s} u)=0,
\end{equation}
where $(-\Delta)^{-s} u$ with $0<s<\min\{1,\dfrac{d}{2}\}$ denotes the inverse of the fractional Laplacian operator.

Using the following JKO scheme 
\begin{eqnarray*}
\label{E-scheme-JKO1}
& u^{k} :=\argmin_{u\in \cP_2(\R^d)}  \bigg\{\frac{1}{2\tau}\mathbf{W}_{2}^2(u,u^{k-1})+ F(u) \bigg\},
\end{eqnarray*}
where $F(u)=\dfrac{1}{2}\|u\|^2_{\dot{H}^{-s}(\R^d)}$ is the square norm of the homogeneous Sobolev space $\dot{H}^{-s}(\R^d)$, Lisini, Mainini and Segatti solved the equation (\ref{E:main}) for $0<s<\min\{1,d/2\}$ \cite{LMS18}. 

Next, we consider the equation
\begin{equation} \label{E:DNS} \partial_t u + \mdiv[h(u) |\xi|^{p-2} \xi] = 0,\end{equation}
where $p>1$, $\xi = -\mdiv \left[\frac{\delta \F}{\delta u}\right]$ and $h: [0,\infty) \to [0,\infty)$ is a nonlinear increasing function. Here $\frac{\delta \F}{\delta u}$ is the first variant of a function $\F:\cP(\R^d)\to \R\cup\{+\infty\}$.
We define the modified Wasserstein distance \cite{DNS} 
$$ \tilde{\mathbf{W}}_q^q (\mu_0,\mu_1) = \inf \bigg\{\int_0^1 \int_{\R^d} h(\rho_t) |\vv_t(x)|^q dxdt: \partial_t \rho_t + \mdiv (h(\rho_t) \vv_t) =0:  \mu_0=\rho_0\mathcal{L}^d , \mu_1=\rho_1\mathcal{L}^d\bigg\},$$
where $q= \frac{p}{p-1}$ and $\mathcal{L}^d$ is the Lebesgue measure on $\R^d$.

Then, the corresponding JKO scheme for \eqref{E:DNS} reads as
$$ u^{k} :=\argmin_{u\in \cP(\R^d)}  \bigg\{\frac{1}{q \tau^{q-1}}\tilde{\mathbf{W}}_{q}^q(u,u^k)+\F(u) \bigg\},$$
where $q= \frac{p}{p-1}$, the H\"older conjugate of $p$. Choosing $h(u) = u^\alpha$ and $\F(u) = \int_{\R^d} |\nabla u|^2 dx$,  (\ref{E:DNS}) becomes the thin-film like equation
$$\partial_t u - \mdiv [u^\alpha \nabla (-\Delta) u] =0.$$

\textbf{Statement of the main results.} In this paper, we would like to develop JKO schemes for modified Wasserstein distances for the following fractional equations. 
  \begin{equation}
  \label{E-main 1}
\partial_t u-\operatorname{div}(u^\alpha\nabla (-\Delta)^{-s}u )=0  \mbox{ in } \mathbb{R}^d.
 \end{equation}
 and 
\begin{equation}
\label{E-main 2} 
 \partial_t u+(-\Delta)^{1-s} u=0 \mbox{ in } \R^d.
 \end{equation}
for $\alpha>0$ and $0<s<\min\{1,\dfrac{d}{2}\}$. 

The first problem \eqref{E-main 1} has been studied by Caffarelli and V\'azquez in \cite{CV1} with $\alpha=1$, where the
existence of solutions was proved for non-negative bounded initial data. Regularity and asymptotic behaviour of these solutions are established in \cite{CV2, CSV, CV3}. Later, the existence of solutions for the equation \eqref{E-main 1} has been solved for the case $0<\alpha<2$    \cite{StTsVz.CRAS, StTsVz16}, and for all $\alpha>0$ \cite{NV, StTsVz18}. The second problem \eqref{E-main 2} was studies in \cite{Erbar}. The proof there consists of introducing a non-local Wasserstein distance and the entropy functionals.  

Our approaches to solve \eqref{E-main 1} and \eqref{E-main 2} are different from \cite{CV1, NV, StTsVz.CRAS, StTsVz16, StTsVz18} and \cite{Erbar} as we use JKO schemes from modified Wasserstein distances and the homogeneous Sobolev norm. With our JKO schemes, we believe that our approaches provide  more efficient computational methods for the problems.

Our JKO scheme is defined inductively as follows. Given a step size $\tau>0$ and a weight function $m:(0,+\infty)\to \R_+$, we start from some $u_0\in \cP(\R^d)$ with $\frac{1}{2}\|u_0\|^2_{\dot{H}^{-s}(\mathbb{R}^d)}<+\infty$ and define
 \begin{equation}
 \label{E-scheme}
 u^0_\tau:=u_0, u^{k}_\tau:=\argmin_{u\in \cP(\R^d)}  \bigg\{\frac{1}{2\tau}\mathbf{W}_{m}^2(u,u^{k-1}_\tau)+\frac{1}{2}\|u\|^2_{\dot{H}^{-s}(\mathbb{R}^d)}\bigg\}, 
  \end{equation}
  where $\mathbf{W}_{m}$ is the modified Wasserstein distance (see section 2) with respect to our chosen weight functions $m$. 

 The approximation $\bar{u}_\tau:[0,\infty)\to \mathcal{P}(\mathbb{R}^d)$ is defined by
  \begin{equation}
 \label{Z12}\bar{u}_\tau(t):=u^k_\tau ~\text{ for }~(k-1)\tau<t\leq k\tau. 
 \end{equation}
 
Now we state our main results. The first one is Theorem \ref{T-solutions of fractional equations} containing all the properties of our gradient flow solutions of the equation \eqref{E-main 1}. 		
  \begin{theorem}
 	\label{T-solutions of fractional equations} Let $u_0\in \mathcal{P}(\mathbb{R}^d)\cap( L^2\cap \dot{H}^{-s})(\mathbb{R}^d)$, $d\geq 2$ and $0<s<\min\{1,\dfrac{d}{2}\}$. 
 	 Then for every $\tau>0$, the scheme (\ref{E-scheme}) always  has a solution $\bar{u}_\tau$ with $m(r)=(r+\tau^{\frac{1}{10}})^\alpha$. In addition to, for every sequence $\tau_n\searrow 0$ there exist a subsequence, still denoted by $\tau_n$, and a function $u:[0,+\infty)\to \cP(\R^d)$ such that 
 	\begin{description}
 		\item[i)]$\bar{u}_{\tau_n}\to u$ strongly in $ L^2(0,T;\dot  H^\beta(\mathbb{R}^d))$, 	$\bar{u}_{\tau_n}^\alpha \, \nabla (-\Delta)^{-s}\bar{u}_{\tau_n} \to u^\alpha\nabla (-\Delta)^{-s}u$  strongly in  $L^1(\mathbb{R}^d\times(0,T))$,   $\bar{u}_{\tau_n}\to u$ weakly in  $L^2(0,T;  \dot{H}^{1-s}(\mathbb{R}^d))$ for every  $T>0,0\leq \beta<1-s$,
    \item [ii)] the function $u$ is a weak solution to the equation
    \begin{equation}
 	\partial_t u-\operatorname{div}(u^\alpha\nabla (-\Delta)^{-s}u )=0
 	\end{equation}
 	in the following weak sense
 	$$\int_0^{+\infty}\int_{\R^d}\partial_t\psi udxdt-\int_0^{+\infty}\int_{\R^d}\nabla_x \psi\cdot u^\alpha\nabla(-\Delta)^{-s}u dxdt=0, \mbox{ }\forall \psi\in C_c^{\infty}((0,+\infty)\times \R^d), $$
 	 		\item [iii)] 	for any $1\leq p\leq 2^L$
 	 		\begin{equation}
 	 		||\bar{u}_\tau(t)||_{L^p}\leq (\lambda(k_0,L)  ( t+k_0 \tau))^{-\frac{(1-1/p)d}{d\alpha+2(1-s)}},
 	 		\end{equation}
 	 		where  
 	 		\begin{equation}
 	 		\lambda(k_0,L) =\frac{c}{\sup_{1\leq n\leq L}\sup _{k\geq k_0}k \left(\left(1+\frac{1}{k-1}\right)^{\frac{(2^n-1)d}{d\alpha+2(1-s)}}-1\right)}
 	 		\end{equation} 
 	 		and 	\begin{equation}
		(\lambda(k_0,L)  \tau)^{-\frac{d}{d\alpha+2(1-s)}}\geq ||u_0||_{L^\infty},
 	 		\end{equation}
 	 			\item [iv)]  for any $t_1\geq 0$ and for any $\Lambda>0$, it holds
 	 			\begin{equation}\label{Z21'}
 	 			\sup_{t\geq t_1}||(\overline{u}_\tau(t)-\Lambda)_+||_{L^2}^2+\int_{t_1+\tau}^{\infty}||(\overline{u}_\tau(t)-\Lambda)_+||_{L^{\frac{(\alpha+2)d}{d-2(1-s)}}}^{\alpha+2} dt\leq  C ||(\overline{u}_\tau(t_1)-\Lambda)_+||_{L^2}^2.
 	 			\end{equation}
 	 			 Moreover, if $\tau<2^{-4}$ and $u(0)\in L^2$, there exists $\Lambda_0=\Lambda_0(||u_0||_{L^2},s,\alpha,d)>0$ such that 
 	 			\begin{equation}\label{Z25'}
 	 			\sup_{t\geq 1}||(\overline{u}_\tau(t)-\Lambda_0)_+||_{L^2}^2+\int_{1}^{\infty}||(\overline{u}_\tau(t)-\Lambda_0)_+||_{L^{\frac{(\alpha+2)d}{d-2(1-s)}}}^{\alpha+2} dt \leq C  \tau^{\frac{d(\alpha+1)}{2(1-s)}+2}.
 	 			\end{equation}
 	 			
 	\end{description}
 \end{theorem}

 Let us describe our strategy to prove Theorem \ref{T-solutions of fractional equations}. First, we establish the regularity of our minimizers $u^k_\tau$. With suitable initial data $u_0$, our minimizers $u^k_\tau$ not only belongs to $\dot{H}^{-s}(\R^d)$ but also $\dot{H}^{1-s}(\R^d)$. Moreover, we can control $\|u^k_\tau\|_{\dot{H}^{1-s}(\R^d)}$ in terms of our potential functions. This task can be achieved as we have a flow interchange as in Lemma \ref{L-flows of modified equations} for our scheme. In order to get our flow interchange, we need to verify that the associated semigroup $S_\delta$ of our auxiliary function $V_\delta$ is a $\lambda$-flow under our weighted Wasserstein distance. To do this, we establish the Eulerian calculus for our weighted Wasserstein metric in Lemma \ref{L-Eulerian calculus}. Note that the flow interchange technique was introduced by McCann, Matthes and Savar\'{e} \cite{MMS}, and the Eulerian calculus for the usual Wasserstein metric was introduced by Otto and Westdickenberg \cite{OttoWest} and was developed later by Daneri and Savar\'{e} \cite{DS}. These ideas have been applied to study PDEs in \cite{LMS12, LMS18} as applications of JKO schemes related to the standard Wasserstein metric. Our achieved regularity estimates of $\{\bar{u}_{\tau_n}\}$ are good enough to obtain a weak solution $u$ of the equation \eqref{E-main 1} via a compactness argument as $\tau_n\to 0$.

The other important features in our Theorem \ref{T-solutions of fractional equations} are the decay rate at infinity of $L^p$ norms and the almost boundedness of solutions $u$ and $u_\tau$ that we stated in  iii) and iv), respectively. The decay rate of $L^p$ norms was already proved in \cite{CSV} and later in \cite{LMS18} for the usual homogeneous porous medium equation with nonlocal pressure with $\alpha=1$. To get iii) and iv) of Theorem \ref{T-solutions of fractional equations}, first we establish $L^{\frac{2d}{d-2(1-s)}}$ estimates of $\cG(u^k_\tau)$ as in Lemma \ref{L-decay of scheme} for our minimizing movements scheme as follows
		\begin{equation}\label{F-Introduction Decay of entropies'}
	C(d,s)\tau||\cG(u^k_\tau)||_{L^{\frac{2d}{d-2(1-s)}}}^2\leq  \int_{\mathbb{R}^d} g(u^{k-1}_\tau(x)) dx-\int_{\mathbb{R}^d} g(u^{k}_\tau(x)) dx, 
	\end{equation}	
	where $\cG(r)=\int_{0}^{r}\sqrt{m(z)g''(z) } dz$ with $g\in C^{2}([0,\infty),\mathbb{R}_+)$ is convex such that $g(0)=g'(0)=g''(0)=0$.
This estimate is obtained by applying the flow interchange technique. Then for the case $p<+\infty$, combining \eqref{F-Introduction Decay of entropies'} for the function $g(z)=z^p$ with interpolation inequalities we get our estimates of the $L^p$-norms. On the other hand, applying \eqref{F-Introduction Decay of entropies'} for the function $g(z)=(z-\Lambda)_+^p$, we will get v) of Theorem \ref{T-solutions of fractional equations}. Then the almost boundedness of solutions $u$ and $u_\tau$ is a consequence of v) by the following argument. 

 Letting $\tau\to 0$, from equation \eqref{Z25'} we obtain 
	\begin{equation}
\lim\limits_{\tau \to 0}\sup_{t\geq 1}||(\overline{u}_\tau(t)-\Lambda_0)_+||_{L^2}^2+\int_{1}^{\infty}||(\overline{u}_\tau(t)-\Lambda_0)_+||_{L^{\frac{(\alpha+2)d}{d-2(1-s)}}}^{\alpha+2} dt =0,
\end{equation}
which implies
\begin{equation}
\sup_{t\geq 1}||(u(t)-\Lambda_0)_+||_{L^2}^2=0.
\end{equation}
That is, $u(t) \leq \Lambda_0$ for all $t\geq 1$.

Let us denote
\begin{equation*}
A_\tau=\{x: (\overline{u}_\tau(t)-\Lambda_0)_+\geq 1, \forall t\geq 1\}.
\end{equation*}
Then, also from equation \eqref{Z25'},
\begin{equation*}
|A_\tau|\leq C \tau^{\frac{d(\alpha+1)}{2(1-s)}+2}.
\end{equation*}
  So, 
\begin{equation}
(\overline{u}_\tau(t)-\Lambda_0)_+\leq 1, i.e. \mbox{ }  {u}_\tau(t)\leq \Lambda_0+1
\end{equation}
 for any $t\geq 1$, and $x\in \mathbb{R}^d\backslash A_\tau$.

 Our second main result is the existence of weak solutions of equation \eqref{E-main 2} 
\begin{theorem}
\label{Z16} Let $u_0\in \mathcal{P}(\mathbb{R}^d)\cap( L^2\cap \dot{H}^{-s})(\mathbb{R}^d)$,  $0< s<\min\{1,d/2\}$. Let $\{\tau_n\}$ be a positive sequence converging to 0, $\overline{u}_\tau$ be the solution of the scheme \eqref{E-scheme}- \eqref{Z12} with $m_n(r)=(r+1)^{\tau_n^{\frac{1}{10}}}$ and $u:[0,+\infty)\to \cP(\R^d)$ as in Lemma \ref{L-the existence of solution}. Then, there exists a subsequence  of $u_n$ still denoted by $u_n$ converging to a solution $u$ of
	\begin{equation}\label{Z18}
	\partial_t u+(-\Delta)^{1-s} u=0.
	\end{equation}
\end{theorem} 

Our paper is organized as follows. In section 2, we review the definition of the modified Wasserstein distance and provide an its approximation. In section 3, we prove the existence and unique of minimizers for solutions of our scheme and then we establish our flow interchange for the scheme via the Eulerian calculus of the modified Wasserstein distances. Finally, in section 4 we will prove our main results.  
\vspace{0.5cm}\\
	\textbf{Acknowledgements:} N-P. Chung is partially supported by the National Research Foundation of Korea (NRF) grants funded by the Korea government No. NRF- 2016R1A5A1008055 and No. NRF-2019R1C1C1007107.  Q-H. Nguyen  is supported by the Academy of Mathematics and Systems Science, Chinese Academy of Sciences startup fund, and the National Natural Science Foundation of China (No. 12050410257 and No. 12288201) and  the National Key R$\&$D Program of China under grant
	2021YFA1000800. We thank Linh Nguyen for useful discussions.  
 \section{Preliminaries}
 First, we review fractional Sobolev spaces. The Fourier transform of $f\in L^1(\R^d)$ is defined by $\hat{f}(\xi):=\int_{\R^d}e^{-ix\cdot\xi} f(x)dx$.
We denote by $\cS(\R^d)$ the Schwartz space of smooth functions on $\R^d$ with rapid decay at infinity, and by $\cS'(\R^d)$ its dual space. For every $r\in \R$, the fractional Sobolev space $H^r(\R^d)$ and the homogeneous fractional Sobolev space $\dot{H}^r(\R^d)$ are defined respectively by
$$ H^r(\R^d):=\bigg\{f\in \cS'(\R^d):\hat{f}\in L^1_{loc}(\R^d), \|f\|_{H^r(\R^d)}:=\frac{1}{(2\pi)^d}\int_{\R^d}(1+|\xi|^{2})^r|\hat{f}(\xi)|^2d\xi<+\infty\bigg\},$$
$$\dot{H}^r(\R^d):=\bigg\{f\in \cS'(\R^d):\hat{f}\in L^1_{loc}(\R^d), \|f\|_{\dot{H}^r(\R^d)}:=\frac{1}{(2\pi)^d}\int_{\R^d}|\xi|^{2r}|\hat{f}(\xi)|^2d\xi<+\infty\bigg\}.$$
If $r<\frac{d}{2}$ then the space $\dot{H}^r(\R^d)$ is a Hilbert space with the scalar product 
$$\langle v,w\rangle_r:=\frac{1}{(2\pi)^d}\int_{\R^d}|\xi|^{2r}\hat{v}(\xi)\overline{\hat{w}(\xi)}d\xi.$$
From \cite[Proposition 1.37]{BCD} we can imply that if $r\in (0,1)$ then there exists a constant $C_{d,r}>0$ such that
\begin{align}
\label{F-scalar prodcut}
\langle v,w\rangle_r=C_{d,r}\int_{\R^d}\int_{\R^d}\frac{(v(x)-v(y))(w(x)-w(y))}{|x-y|^{d+2r}}, \mbox{  }\forall u,v\in \dot{H}^r(\R^d). 
\end{align}
If $r\in (0,d/2)$ then by \cite[Theorem 1.38]{BCD} there exists $C_{d,r}>0$ such that following Sobolev inequality holds 
\begin{align}
\label{I-Sobolev embedding}
\|u\|_{L^q(\R^d)}\leq C_{d,r}\|u\|_{\dot{H}^r(\R^d)},
\end{align}
for every $u\in \dot{H}^r(\R^d)$, where $q:=\frac{2d}{d-2\beta}$.

Given $0<s<\min\{1,\dfrac{d}{2}\}$, the $s$-fractional Laplician $(-\Delta)^s$ on $\R^d$ is defined by means of Fourier transform as
$$(\widehat{(-\Delta)^s u})(x)=|x|^{2s}\widehat{u}(x), \mbox{ for every } x\in \R^d.$$

Now let us review the definition of the modified Wasserstein distances introduced in \cite{DNS} and studied in \cite{CLSS}. We denote by  $\Pn(\R^d)$ the space of probability Borel measures on $\R^d$ and $\Pn_2(\R^d)$ the space of $\mu\in  \Pn(\R^d)$ with finite second order moment. Let $\mathcal{L}^d$ be the Lebesgue measure on $\R^d$. 

The Wasserstein distance on $\cP_2(\R^d)$ is defined by
 	\begin{equation*}
 	\mathbf{W}_2(\mu^0,\mu^1)=\min_{\pi} \left\{\int_{\mathbb{R}^d\times \mathbb{R}^d}|x-y|^2 d\pi(x,y):\pi\in \Gamma(\mu^0,\mu^1) \right\}^{1/2},
 	\end{equation*}
where
\begin{equation*}
 	\Gamma(\mu^0,\mu^1)=\left\{\pi\in \mathcal{P}(\mathbb{R}^d\times\mathbb{R}^d): \pi(A\times\mathbb{R}^d)=\mu^0(A),\pi(\mathbb{R}^d\times A)=\mu^1(A)~\forall~\text{Borel}~A\subset \R^d\right\}.
 	\end{equation*}
 The Wasserstein distance also admits a dynamical system description as follows \cite{BB}:
 $$\W_2^2 (\mu_0,\mu_1) = \inf \bigg\{\int_0^1 \int_{\R^d} \rho_s(x) \,  |\vv_s(x)|^2 dxdt: \partial_s \rho_s + \mdiv (\rho_s \vv_s) =0,   \mu_0=\rho_0 \mathcal{L}^d, \mu_1=\rho_1 \mathcal{L}^d\bigg\}.$$

  Let $X$ be a  metric space, we denote by $\cM^+(X)$ the space of all finite nonnegative Borel measures on $X$. We denote by $\cM(X,\R^d)$ the space of $\R^d$-valued Borel measures on $X$ with finite total variation.
  
  For two nonnegative functions $f,g$, our notation $f\sim g$ means that there exist $C_1,C_2>0$ such that $C_1f\leq g\leq C_2f.$

 We denote by $\CE$ the set of $(\mu_t)_{t\in [0,1]}\subset \cM^+(\mathbb{R}^d)$ and $(\nu_s)_{s\in [0,1]}\subset \cM(\mathbb{R}^d,\R^d)$ such that 
 \begin{enumerate}
 	\item $t\mapsto \mu_t$ is weakly* continuous in $\cM(\mathbb{R}^d)$;
 	\item $t\mapsto \nu_t$ is Borel and $\int_0^1|\nu_t|(\mathbb{R}^d)dt<\infty$;
 	\item $(\mu_t,\nu_t)_{t\in [0,1]}$ is a distributional solution of the following continuity equation(CE)
 	$$\partial_t\mu_t+\nabla\cdot \nu_t=0 \mbox{ in } \R^d\times (0,1),$$
which means $$\int_0^1\int_{\R^d}\partial_t\varphi(x,t)d\mu_t(x)dt+\int_0^1\int_{\R^d}\nabla_x \varphi(x,t)\cdot d\nu_t(x)dt=0,$$ for every $\varphi\in C_c^1(\R^d\times (0,1)).$
 \end{enumerate}
 Given $\nu^0,\nu^1\in \cM^+(\mathbb{R}^d)$, we denote by $\CE(\nu^0\to \nu^1)$ the subset of $\CE$ such that $\mu_0=\nu^0,\mu_1=\nu^1.$
  
 Given $\varepsilon>0$ and $\alpha\in (0,1]$, the action density function $G:[0,\infty)\times \R^d\to [0,+\infty)$ is defined by $$G(s_1,s_2):=\frac{|s_2|^2}{m(s_1)} \mbox{ for every } (s_1,s_2)\in [0,\infty)\times \R^d,$$
 where 
 \begin{equation*}
 m(s)\sim (s+\varepsilon)^\alpha.
 \end{equation*}  
  Given measures $\mu\in \cM^+(\R^d)$ and $\nu\in \cM(\R^d,\R^d)$, we consider the Lebesgue decompositions $\mu=\rho \mathcal{L}^d+\mu^\perp$ and $\nu=w\mathcal{L}^d+\nu^\perp$. Here, $\rho\in L^1_+(\R^d, \mathcal{L}^d;\R)$, $w\in L^1(\R^d,\mathcal{L}^d;\R^d)$, and $\nu^\perp$ is the singular part of $\nu$ with respect to the Lebesgue measure $\mathcal{L}^d$ of $\R^d$. We define $\cG(\mu,\nu)$ as follows:  
 	\[ \cG(\mu,\nu):=\begin{cases} 
 	\int_{\R^d} G(\rho,\omega)dx  \mbox{ if }\nu^\perp=0, \\
 	+\infty  \mbox{ otherwise}; 
 	\end{cases}
 	\]
 
 Now we are ready to define the modified Wasserstein distance. Given $\mu^0,\mu^1\in \cM^+(\mathbb{R}^d)$ we define
 \begin{align*}
 \mathbf{W}_{m}(\mu^0,\mu^1):&=\inf\bigg\{\bigg(\int_0^1\cG(\mu_t,\nu_t)dt\bigg)^{1/2}:(\mu,\nu)\in \CE(\mu^0\to\mu^1)\bigg\}\\
 &=\inf\bigg\{\int_0^1(\cG(\mu_t,\nu_t))^{1/2}dt:(\mu,\nu)\in \CE(\mu^0\to\mu^1)\bigg\}.
 \end{align*}
 
 The equality in the above definition follows from \cite[Theorem 5.4]{DNS}. If the set $\CE(\mu^0\to\mu^1)$ is empty we put $\mathbf{W}_{m}(\mu^0,\mu^1)=+\infty$. 

The following lemma provides an approximation of the modified Wasserstein distance.
\begin{lemma} \label{apro} Let $m\in C^\infty (\mathbb{R}_+,\mathbb{R}_+)$ be such that $\inf_{x\in \mathbb{R}_+} m(x)>0$. Let $\mu^0,\mu^1\in \mathcal{P}(\mathbb{R}^d)\cap L^1(\mathbb{R}^d)$ be such that $\mathbf{W}_{m}(\mu^0,\mu^1)<\infty$. Then there exist $\rho_n\in C^\infty_c([0,1]\times\mathbb{R}^d)$ and $\phi_n\in C^\infty([0,1]\times \mathbb{R}^d)\cap L^\infty([0,1],H^2(\mathbb{R}^d))$ such that 
	\begin{itemize}
		\item $\rho_n(t)\in \mathcal{P}(\mathbb{R}^d)$ for every $t\in [0,1]$, 
		 $||\rho_n(0)-\mu^0||_{L^1(\mathbb{R}^d)}+||\rho_n(1)-\mu^1||_{L^1(\mathbb{R}^d)}\to 0$ as $n\to \infty,$
		\item $(\rho_n,\phi_n)$ satisfies $\partial_t\rho_{n}(t,x)=-\operatorname{div}(m(\rho_n(t,x))\nabla\phi_n(t,x))$ and 
		\begin{equation}
	\mathbf{W}_{m}^2(\mu^0,\mu^1)=	\lim_{n\to +\infty}\int_{0}^{1}\int_{\mathbb{R}^d}m(\rho_n(t,x))|\nabla_x \phi_n(t,x)|^2 dxdt.
		\end{equation}
	\end{itemize}  
\end{lemma}
\begin{proof} We can choose  $\rho_n(0),\rho_{n}(1)\in  \mathcal{P}(\mathbb{R}^d)\cap C^\infty_c(\mathbb{R}^d)$ such that $||\rho_n(0)-\mu^0||_{L^1(\mathbb{R}^d)}+||\rho_n(1)-\mu^1||_{L^1(\mathbb{R}^d)}\to 0$ as $n\to \infty,$ and 
$
		\mathbf{W}_{m}^2(\mu^0,\mu^1)=\lim\limits_{n\to \infty}	\mathbf{W}_{m}^2(\rho_{n}(0),\rho_{n}(1)).$\\
	For every $n\in \N$, there exist $\rho_{n}\in C^\infty_c([0,1]\times \mathbb{R}^d)$ and $\nu_n\in C^\infty_c([0,1]\times \mathbb{R}^d, \mathbb{R}^d)$ such that $\partial_t\rho_n(t,x)=-\operatorname{div}(\nu_n(t,x))$ and 
	\begin{equation}\label{Z23}
\int_{0}^{1}\int_{\mathbb{R}^d}\frac{|\nu_n(t,x)|^2}{m(\rho_n(t,x))} dxdt\leq 	\mathbf{W}_{m}^2(\rho_{n}(0),\rho_{n}(1))+\frac{1}{n}.
	\end{equation}
	Let $\phi_{n}(t,x)$ be the smooth solution to $\operatorname{div}(m(\rho_n(t,x))\nabla\phi_n(t,x))=\operatorname{div}(\nu_n(t,x))$.  Using $\phi_{n}$ as test function of this equation to get,
	\begin{equation*}
	\int_{\mathbb{R}^d} m(\rho_n(t,x))|\nabla\phi_n(t,x)|^2dx=\int_{\mathbb{R}^d} \nu_n(t,x)\nabla \phi_{n}(t,x)dx.
	\end{equation*}
	By H\"{o}lder inequality, 
		\begin{equation*}
	\int_{\mathbb{R}^d} m(\rho_n(t,x))|\nabla\phi_n(t,x)|^2dx\leq \left(\int_{\mathbb{R}^d}\frac{|\nu_n(t,x)|^2}{m(\rho_n(t,x))} dx\right)^{1/2}\left(\int_{\mathbb{R}^d} m(\rho_n(t,x))|\nabla\phi_n(t,x)|^2dx\right)^{1/2}.
	\end{equation*}
Combining this with \eqref{Z23} and definition of $\mathbf{W}_{m}^2(\rho_{n}(0),\rho_{n}(1))$, one obtains 
	\begin{equation*}
\mathbf{W}_{m}^2(\rho_{n}(0),\rho_{n}(1))\leq \int_{0}^{1}	\int_{\mathbb{R}^d} m(\rho_n(t,x))|\nabla\phi_n(t,x)|^2dx dt\leq 	\mathbf{W}_{m}^2(\rho_{n}(0),\rho_{n}(1))+\frac{1}{n}.
	\end{equation*}
	Letting $n\to \infty$, we obtain the result. The proof is complete. 
\end{proof}\\ 
\section{JKO schemes, flow interchange and Eulerian calculus}

 We begin this section with the following lemma proving the uniqueness solution of the scheme (\ref{E-scheme}).
\begin{lemma}
	\label{L-solutions of the scheme}
	For every $\tau>0$, every $u_0\in  \cP_2(\R^d)\cap \dot{H}^{-s}(\R^d)$ and every $k\in \N$, the scheme (\ref{E-scheme}) has a
	unique solution $u^k_\tau$.
\end{lemma}
\begin{proof}
By \cite[Theorem 5.5 and Theorem 5.6] {DNS}, and \cite[Proposition 3.1]{LMS18}, the function $u\mapsto 
\frac{1}{2\tau}W^2_m(u,u^{k-1}_\tau) + \frac{1}{2}\|u\|^2_{\dot{H}^{-s}(\R^d)}$
is nonnegative, lowersemicontinuous with respect to the weak* topology and has weak* compact sublevels. Therefore , the scheme \eqref{E-scheme} has minimizers. As $W_m^2(\cdot,\cdot )$ is convex \cite[Theorem 5.11]{DNS} and the map $u\mapsto \frac{1}{2}\|u\|^2_{\dot{H}^{-s}(\R^d)}$
is strictly convex we get
that the map $\frac{1}{2\tau}W^2_m(u,u^{k-1}_\tau) + \frac{1}{2}\|u\|^2_{\dot{H}^{-s}(\R^d)}$
is strictly convex and hence we obtain the
uniqueness of minimizers.
\end{proof}

 Let $U:[0,+\infty)\to [0,+\infty)$ be the function satisfying $U''(s):=\frac{1}{m(s)}$ with $U'(0)=U(0)=0$. We define 
$$\mathbf{U}(u):=\int_{\mathbb{R}^d} U(u(x))dx \mbox{ for every } u\in \mathcal{P}(\mathbb{R}^d).$$
Since $m(s)\sim (s+\varepsilon)^{\alpha}$,  one has $U(s)\sim \min\{\varepsilon^{-\alpha} s^2,s^{2-\alpha}\}$. Thus, 
\begin{equation}
\label{I-of U}
\mathbf{U}(u) \sim \int \min\{\varepsilon^{-\alpha}u(x)^2, u(x)^{2-\alpha}\} dx.
\end{equation}
Since $U$ is convex, the function $\mathbf{U}$ is lower semicontinuous with respect to the weak convergence on $\mathcal{P}(\mathbb{R}^d)$.

Let $\varphi\in C^\infty_c(\mathbb{R}^d)$ be a test function and let $\delta>0$. Let $\mathbf{S}_\delta$ be the semigroup defined by $\mathbf{S}_{\delta,t}v_0=v_t$ for every $t>0$ with $v_t$ is the unique distribution solution, i.e. $v_t\in \cP(\R^d)$, of the following equation with initial data $v_0\in \cP(\R^d)$ 
\begin{equation}
 		\label{E-solutions of approximations}
 		\partial_t v_t-\dive (m(v_t) \nabla\varphi)-\delta \Delta v_t=0 \mbox{ in } (0,+\infty)\times \mathbb{R}^d.
 		\end{equation}  
 Let $\rho\in C_c^\infty([0,1]\times \R^d)$ be such that $\rho(t)\in \cP(\R^d)$ for every $t\in [0,1]$. For every $h,t>0$ we put 
  $$\rho^h(t)=\mathbf{S}_{\delta,ht}\rho(t) \in \mathcal{P}(\mathbb{R}^d).$$
Let $\phi^h$ be the unique smooth solution to
\begin{equation}\label{z1}\partial_t\rho^h(t,x)=-\operatorname{div}(m(\rho^h(t,x))\nabla\phi^h(t,x)) \mbox{ in } [0,1]\times\mathbb{R}^d.
\end{equation}
We define 
\begin{equation*}
\mathbf{A}^h(t)=\int_{\mathbb{R}^d} m(\rho^h(t,x))|\nabla \phi^h(t,x)|^2 dx.
\end{equation*}

To prove that $S_\delta$ is a $\lambda$-flow in our modified Wasserstein distance we need the following lemma which is the Eulerian calculus in our metric. Its proof is an adaptation of the proof of \cite[formula 76]{LMS12}.
\begin{lemma}
\label{L-Eulerian calculus}
Let $m(r)=(r+\tau^{1/10})^\alpha$ or $m(r)=(r+1)^{\tau^{\frac{1}{10}}}$. Let $\mathbf{V}_{\delta}:\mathcal{P}(\mathbb{R}^d)\to \R$ be the function defined by
 		\begin{equation}\label{Z14}
 		\mathbf{V}_\delta(v):=\langle \varphi, v\rangle+\delta \mathbf{U}(v) \mbox { for every } v\in \mathcal{P}(\mathbb{R}^d).
 		\end{equation}
Then for every $t\in [0,1]$ and $h\geq 0$, we have  	 
\begin{equation}\label{interch}
\frac{1}{2}\partial_h \mathbf{A}^h(t) +\partial_t \mathbf{V}_\delta(\rho^h(t))\leq -\lambda_\delta  t \mathbf{A}^h(t),
\end{equation}
where \begin{equation*}
\lambda_{\delta}=-\frac{1}{2\delta}||\nabla\varphi||_{L^\infty}^2\sup_{a>0} (m(a)|m''(a)|)-||D^2\varphi||_{L^\infty} \sup_{a>0}|m'(a)|.
\end{equation*}
\end{lemma}
\begin{proof}
First, we have 
\begin{equation}\label{Z6}
\partial_h \rho^h(t)=t(\partial_z \mathbf{S}_{\delta,z})_{z=ht}(\rho(t))=t\dive (m( \rho^h(t)) \nabla\varphi)+\delta t \Delta \rho^h(t).
\end{equation}
Hence, thanks to \eqref{z1}
\begin{align}
\partial_t\partial_h \rho^h(t)&=\nonumber
\delta \Delta \rho^h(t)+\dive (m( \rho^h(t)) \nabla\varphi)-t\dive \left[ \operatorname{div}(m(\rho_n^h(t))\nabla\phi^h(t))m'( \rho^h(t)) \nabla\varphi\right]\\&~~~-\delta t \Delta\left[ \operatorname{div}(m(\rho^h(t))\nabla\phi^h(t))\right].\label{Z7}
\end{align}
For every $t\in [0,1]$ and $h\geq 0$, we have
\begin{align*}
&\partial_t \mathbf{V}_\delta(\rho^h(t))=\langle \varphi, \partial_t\rho^h(t)\rangle+\delta\int_{\mathbb{R}^d} U'(\rho^h(t))\partial_t\rho^h(t)\\&~~~=- \langle \varphi, \operatorname{div}(m(\rho^h(t))\nabla\phi^h(t))\rangle-\delta\int_{\mathbb{R}^d} U'(\rho^h(t)) \operatorname{div}(m(\rho_n^h(t))\nabla\phi_n^h(t))
\\&~~~= \int_{\R^d}  m(\rho^h(t))\nabla \varphi \nabla\phi^h(t)+\delta\int_{\mathbb{R}^d} U^{''}(\rho^h(t)) m(\rho^h(t))\nabla\phi^h(t)\nabla \rho^h(t).
\end{align*}
Since $U^{''}(\rho^h_n(t)) m(\rho_n^h(t))=1$, we get
\begin{equation}\label{Z8}
\partial_t \mathbf{V}_\delta(\rho^h(t))= \int_{\R^d}  m(\rho^h(t))\nabla \varphi \nabla\phi^h(t)-\delta\int_{\mathbb{R}^d} \rho^h(t)\Delta\phi^h(t).
\end{equation}
\newline
On the other hand,
\begin{align}\nonumber
\frac{1}{2}\partial_h \mathbf{A}^h(t)&=-\frac{1}{2}\int_{\R^d} \partial_hm(\rho^h) |\nabla \phi^h|^2+\int_{\R^d} \partial_h\left( m(\rho^h) \nabla \phi^h\right) \nabla \phi^h\\&\overset{\eqref{z1}}=-\frac{1}{2}\int_{\R^d} \partial_h\rho^h m'(\rho^h) |\nabla \phi^h|^2+\int_{\R^d} (\partial_t\partial_h\rho^h) \phi^h. \label{Z1}
\end{align}
By \eqref{Z6}, we obtain
\begin{align}\nonumber
&-\frac{1}{2}\int_{\R^d} \partial_h\rho^h m'(\rho^h) |\nabla \phi^h|^2=-\frac{t}{2}\int_{\R^d}\left[\dive (m( \rho^h(t)) \nabla\varphi)+\delta  \Delta \rho^h(t)\right] m'(\rho^h) |\nabla \phi^h|^2\\
&\nonumber=\frac{t}{2} \int_{\R^d}\left[m( \rho^h(t)) \nabla\varphi+\delta \nabla \rho^h(t)\right]\nabla\left( m'(\rho^h) |\nabla \phi_{n}^h|^2\right)\\
&\nonumber=\frac{t\delta}{2}\int_{\R^d} m''(\rho^h) | \nabla \rho^h(t)|^2 |\nabla \phi^h|^2+\frac{t\delta}{2}\int_{\R^d}  m'(\rho^h) \nabla \rho^h(t)\nabla (|\nabla \phi^h|^2)\\&\nonumber+\frac{t}{2}\int_{\R^d} m( \rho^h(t))m''( \rho^h(t)) \nabla\varphi\nabla \rho^h  |\nabla \phi^h|^2+t\int_{\R^d} m( \rho^h(t))m'( \rho^h(t)) \nabla \varphi  \nabla (\frac{1}{2}|\nabla \phi^h|^2)
\\&\nonumber=\frac{t\delta}{2}\int_{\R^d} m''(\rho^h) | \nabla \rho^h(t)|^2 |\nabla \phi^h|^2-t\delta\int_{\R^d}  m(\rho^h) \Delta (\frac{1}{2}|\nabla \phi^h|^2)\\&+\frac{t}{2}\int_{\R^d} m( \rho^h(t))m''( \rho^h(t)) \nabla\varphi\nabla \rho^h  |\nabla \phi^h|^2+t\int_{\R^d} m( \rho^h(t))m'( \rho^h(t)) \nabla \varphi  \nabla (\frac{1}{2}|\nabla \phi^h|^2).\label{z2}
\end{align}
It follows from \eqref{Z7} that
\begin{align*}\nonumber
&\int_{\R^d} (\partial_t\partial_h\rho^h) \phi^h=
\delta\int_{\R^d}  \rho^h(t) \Delta \phi^h-\int_{\R^d} m( \rho^h(t)) \nabla\varphi\nabla \phi^h(t)\\&\nonumber
+t\int_{\R^d} \operatorname{div}(m(\rho^h(t))\nabla\phi^h(t)))m'( \rho^h(t)) \nabla\varphi\nabla \phi^h-\delta t \int_{\R^d} \operatorname{div}(m(\rho^h(t)\nabla\phi^h(t))) \Delta \phi^h\\&\nonumber
\overset{\eqref{Z8}}= -\partial_t \mathbf{V}_\delta(\rho^h(t))-t\int_{\R^d} m(\rho^h(t))\nabla\phi^h(t)\nabla \left[m'( \rho^h(t)) \nabla\varphi\nabla \phi^h\right]+\delta t \int_{\R^d} m(\rho^h(t))\nabla\phi^h(t) \nabla(\Delta \phi^h)
\\&\nonumber
= -\partial_t \mathbf{V}_\delta(\rho^h(t))-t\int_{\R^d} m(\rho^h(t))m'(\rho^h(t))\nabla\phi^h(t)\nabla (\nabla\varphi\nabla \phi^h)\\&-t\int_{\R^d} m(\rho^h(t))m''(\rho^h(t))(\nabla\phi^h(t)\nabla \rho^h) ( \nabla\varphi\nabla \phi^h)+\delta t \int_{\R^d} m(\rho^h(t))\nabla\phi^h(t) \nabla(\Delta \phi^h).
\end{align*} 
Combining this with \eqref{z2}, \eqref{Z1} and \eqref{Z8}, one finds 
\begin{align*}
&\frac{1}{2}\partial_h \mathbf{A}^h(t)+\partial_t \mathbf{V}_\delta(\rho^h(t))\\&=\frac{t\delta}{2}\int_{\R^d} m''(\rho^h) | \nabla \rho^h(t)|^2 |\nabla \phi^h|^2+t\delta\int_{\R^d}  m(\rho^h) \left[-\Delta (\frac{1}{2}|\nabla \phi^h|^2)+\nabla\phi^h(t) \nabla(\Delta \phi^h)\right]\\
&+\frac{t}{2}\int_{\R^d} m( \rho^h(t))m''( \rho^h(t)) \left(\nabla\varphi\nabla \rho^h  |\nabla \phi^h|^2-2(\nabla\phi^h(t)\nabla \rho^h) ( \nabla\varphi\nabla \phi^h)\right)\\
&+t\int_{\R^d} m( \rho^h(t))m'( \rho^h(t)) \left(\nabla \varphi \nabla (\frac{1}{2}|\nabla \phi^h|^2)-\nabla\phi^h(t)\nabla (\nabla\varphi\nabla \phi^h)\right).
\end{align*}
Since $m''\leq 0$ and 
\begin{align*}
&-\Delta (\frac{1}{2}|\nabla \phi^h|^2)+\nabla\phi^h(t) \nabla(\Delta \phi^h)=-|\nabla^2 \phi^h|^2 \mbox{ (the Bochner formula)},
\\&
\left|\nabla\varphi\nabla \rho^h  |\nabla \phi^h|^2-2(\nabla\phi^h(t)\nabla \rho^h) ( \nabla\varphi\nabla \phi^h)\right|\leq 2 |\nabla \varphi| |\nabla \phi^h|^2|\nabla \rho^h|,\\
& |\nabla \varphi \nabla (\frac{1}{2}|\nabla \phi^h|^2)-\nabla\phi^h(t)\nabla (\nabla\varphi\nabla \phi^h)|\leq |\nabla^2\varphi||\nabla\phi^h|^2, 
\end{align*}
one has 
\begin{align*}
&\frac{1}{2}\partial_h \mathbf{A}^h(t)+\partial_t \mathbf{V}_\delta(\rho^h(t))\leq -\frac{t\delta}{2}\int_{\R^d} |m''(\rho^h)| | \nabla \rho^h(t)|^2 |\nabla \phi^h|^2\\
&+t\int_{\R^d} \left(m( \rho^h(t))|m''( \rho^h(t))|^{1/2} |\nabla \varphi| |\nabla \phi^h|\right)\left( |m''( \rho^h(t))|^{1/2}  |\nabla \phi^h||\nabla \rho^h|\right)\\
&+t\int_{\R^d} m( \rho^h(t))|m'( \rho^h(t))|  |D^2\varphi||\nabla\phi^h|^2\\
&\leq 
\frac{t}{2\delta}\int_{\R^d} |m( \rho^h(t))|^2m''( \rho^h(t))| |\nabla \varphi|^2 |\nabla \phi^h|^2+t\int_{\R^d} m( \rho^h(t))|m'( \rho^h(t))|  |D^2\varphi||\nabla\phi^h|^2.
\end{align*}
Thus, we deduce that 
\begin{align*}
&\frac{1}{2}\partial_h \mathbf{A}^h(t)+\partial_t \mathbf{V}_\delta(\rho^h(t))\leq -\lambda_{\delta} t \mathbf{A}^h(t)
\end{align*}
with 
\begin{equation}
\lambda_{\delta}=-\frac{1}{2\delta}||\nabla\varphi||_{L^\infty}^2\sup_{a\geq 0} (m(a)|m''(a)|)-||D^2\varphi||_{L^\infty} \sup_{a\geq 0}|m'(a)|.
\end{equation}
The proof is complete. 
\end{proof}
	\vspace{0.5cm}\\
Now in the following lemma we are ready to establish our crucial flow interchange of our scheme via proving that $S_\delta$ is a $\lambda$-flow in the modified Wasserstein distance. 
     \begin{lemma}
 	\label{L-flows of modified equations}  Let $\varphi\in C^\infty_c(\mathbb{R}^d)$ be a test function. Let $m(r)=(r+\tau^{1/10})^\alpha$ or $m(r)=(r+1)^{\tau^{\frac{1}{10}}}$. Let $\delta>0$, we define the function $\mathbf{V}_{\delta}:\mathcal{P}(\mathbb{R}^d)\to \R$ by
 		\begin{equation}\label{Z14}
 		\mathbf{V}_\delta(v):=\langle \varphi, v\rangle+\delta \mathbf{U}(v) \mbox { for every } v\in \mathcal{P}(\mathbb{R}^d).
 		\end{equation}
 	 Then there holds 
 	 	\begin{equation}\label{Z13}
 	 \mathbf{V}_\delta(u^k_\tau)-\mathbf{V}_\delta(u^{k-1}_\tau)\leq  -\tau\delta||u^k_\tau||_{\dot{H}^{1-s}}^2+\tau\langle \dive \left(m(u^k_\tau)\nabla (-\Delta) ^{-s}u^k_\tau\right),\varphi \rangle-\frac{\lambda_\delta}{2}\mathbf{W}^2_{m}(u^k_\tau,u^{k-1}_\tau)
 	 \end{equation}
where 
\begin{equation*}
\lambda_{\delta}=-\frac{1}{2\delta}||\nabla\varphi||_{L^\infty}^2\sup_{a>0} (m(a)|m''(a)|)-||D^2\varphi||_{L^\infty} \sup_{a>0}|m'(a)|.
\end{equation*}
In particular, 
	\begin{equation}\label{Z9}
\mathbf{V}_\delta(u^k_\tau)-\mathbf{V}_\delta(u^{k-1}_\tau)\leq  \tau\langle \dive \left(m(u^k_\tau)\nabla (-\Delta) ^{-s}u^k_\tau\right),\varphi \rangle-\frac{\tau\lambda_\delta}{2}\left(||u^{k-1}_\tau||_{\dot{H}^{-s}}^2-||u^{k}_\tau||_{\dot{H}^{-s}}^2\right).
\end{equation}
 \end{lemma}
 \begin{proof} First, \eqref{Z9} follows from \eqref{Z13} and the fact that  
 		$$\frac{1}{2\tau} \mathbf{W}^2_{m}(u^k_\tau,u^{k-1}_\tau)+\frac{1}{2}||u^k_\tau||_{\dot{H}^{-s}}^2\leq \frac{1}{2}||u^{k-1}_\tau||_{\dot{H}^{-s}}^2.$$
 		We now prove \eqref{Z13}. We consider the equation 
 		\begin{equation}
 		\label{E-solutions of approximations'"}
 		\partial_t v_t-\dive (m(v_t) \nabla\varphi)-\delta \Delta v_t=0 \mbox{ in } (0,+\infty)\times \mathbb{R}^d.
 		\end{equation}
 		Clearly, the equation \eqref{E-solutions of approximations'"} has the uniqueness solution with initial data $v_0\in \mathcal{P}(\mathbb{R}^d)$. \\
 		We define the semigroup $\mathbf{S}_{\delta}$ by $\mathbf{S}_{\delta,t}v_0=v_t$ for every $t>0$.
 	By the definition of $u_\tau^k$ as a minimizer of \eqref{E-scheme}: 
 	$$\frac{1}{2\tau} \mathbf{W}^2_{m}(u^k_\tau,u^{k-1}_\tau)+\frac{1}{2}||u^k_\tau||_{\dot{H}^{-s}}^2\leq \frac{1}{2\tau} \mathbf{W}^2_{m}(\mathbf{S}_{\delta,h}(u^k_\tau),u^{k-1}_\tau)+\frac{1}{2}||\mathbf{S}_{\delta,h}(u^k_\tau)||_{\dot{H}^{-s}}^2,$$
 	and 
 		\begin{align*}
 	\lim_{h\to 0}\frac{||\mathbf{S}_{\delta,h} (u^k_\tau)||_{\dot{H}^{-s}}^2-||(u^k_\tau)||_{\dot{H}^{-s}}^2}{2h}&=\langle(-\Delta) ^{-s} u^k_\tau,(\partial_h\mathbf{S}_{\delta,h} (u^k_\tau))_{h=0}\rangle \\&\overset{\eqref{E-solutions of approximations'"}}= \langle(-\Delta) ^{-s} u^k_\tau,\dive (m(u^k_\tau) \nabla\varphi)+\delta \Delta u^k_\tau\rangle\\&
 	=-\delta||u^k_\tau||_{\dot{H}^{1-s}}^2+\langle \dive \left(m(u^k_\tau)\nabla (-\Delta) ^{-s}u^k_\tau\right),\varphi \rangle.
 	\end{align*}
 	So, it is enough to show that 
 \begin{equation}\label{Z2}
\frac{1}{2}\limsup_{h\to 0}\frac{\mathbf{W}^2_{m}(\mathbf{S}_{\delta,h}(\xi),\mu)^2-\mathbf{W}^2_{m}(\xi,\mu)^2}{h}+\frac{\lambda_\delta}{2}\mathbf{W}^2_{m}(\mu,\xi)^2+ \mathbf{V}_\delta(\xi)\leq  \mathbf{V}_\delta(\mu)
\end{equation} 
for any $\mu,\xi\in \mathcal{P}(\mathbb{R}^d)$  with $ \mathbf{V}_\delta(\mu),\mathbf{V}_\delta(\xi),\mathbf{W}^2_{m}(\mu,\xi)<\infty.$\\
Let $\rho_{n}$ be in Lemma \ref{apro} with $\mu^0=\mu,\mu^1=\xi$.
Set for $h>0, t>0$, $$\rho^h_n(t)=\mathbf{S}_{\delta,ht}v_0 \in \mathcal{P}(\mathbb{R}^d),$$
with $v_0=\rho_{n}(t).$\newline
\\Let $\phi_{n}^h$ be a unique solution to 
$$\partial_t\rho_{n}^h(t,x)=-\operatorname{div}(m(\rho_n^h(t,x))\nabla\phi_n^h(t,x)) \mbox{ in } [0,1]\times\mathbb{R}^d.$$
 Set 
\begin{equation*}
\mathbf{A}_n^h(t)=\int_{\mathbb{R}^d} m(\rho_n^h(t,x))|\nabla \phi_{n}^h(t,x)|^2 dx.
\end{equation*}
We have 
\begin{equation}\label{Z5}
\mathbf{W}^2_{m}(\mu,\xi)^2=\lim\limits_{n\to\infty}\int_{0}^{1} \mathbf{A}_n^0(t)dt.
\end{equation}
From Lemma \ref{L-Eulerian calculus} we get that
\begin{equation}\label{interch}
\frac{1}{2}\partial_h \mathbf{A}_n^h(t) +\partial_t \mathbf{V}_\delta(\rho^h_n(t))\leq -\lambda_\delta  t \mathbf{A}_n^h(t)
\end{equation}
for every $n\in \N,t\in [0,1]$ and $h\geq 0$. 
 
 Now, we show that \eqref{interch} implies \eqref{Z2}. Indeed, \eqref{interch} gives for $
\mathbf{c}=||\varphi||_{L^\infty}$
\begin{align*}
\frac{1}{2}\partial_h\int_{0}^{1}e^{2\lambda_{\delta}t h}\mathbf{A}_n^h(t) dt&\leq -\int_{0}^{1}e^{2\lambda_{\delta}t h}\partial_t( \mathbf{V}_\delta(\rho^h_n(t)) +\mathbf{c})dt\\&= \mathbf{V}_\delta(\rho_n(0))+\mathbf{c}- e^{2\lambda_{\delta}h}\left(\mathbf{V}_\delta(\rho^h_n(1))+\mathbf{c}\right)+2\lambda_{\delta} h\int_{0}^{1}e^{2\lambda_{\delta}t h} (\mathbf{V}_\delta(\rho^h_n(t))+\mathbf{c}) dt\\&\leq 
\mathbf{V}_\delta(\rho_n(0))+\mathbf{c}- e^{2\lambda_{\delta}h}\left(\mathbf{V}_\delta(\mathbf{S}_{\delta,h}\rho_{n}(1))+\mathbf{c}\right),
\end{align*}
here we have used the fact that  $\mathbf{V}_\delta(\rho^h_n(t))+\mathbf{c}\geq 0$ and $\lambda_\delta<0$. \\
So, 
\begin{equation}\label{Z3}
\frac{1}{2}\int_{0}^{1}e^{2\lambda_{\delta}t h}\mathbf{A}_n^h(t) dt\leq \frac{1}{2}\int_{0}^{1}\mathbf{A}_n^0(t) dt+
h(\mathbf{V}_\delta(\rho_n(0))+\mathbf{c})- \frac{1-e^{2\lambda_{\delta}h}}{-2\lambda_{\delta}}\left(\mathbf{V}_\delta(\mathbf{S}_{\delta,h}\rho_{n}(1))+\mathbf{c}\right),
\end{equation}
since the map $h\mapsto \mathbf{V}_\delta(\mathbf{S}_{\delta,h}\rho_{n}(1))+\mathbf{c}$ is decreasing. \\
For every positive function $\theta\in C^1([0,1])$, we have 
 $$\partial_t\rho_{n}^h(\tilde{\theta}^{-1}(t),x)=-\operatorname{div}(m(\rho_n^h(\tilde{\theta}^{-1}(t),x))\nabla ((\tilde{\theta}^{-1})'(t)\phi_n^h(\tilde{\theta}^{-1}(t),x))),$$
 and 
$$
 \rho_{n}^h(\tilde{\theta}^{-1}(0),x)=\rho_{n}^h(0,x),~~ \rho_{n}^h(\tilde{\theta}^{-1}(1),x)=\rho_{n}^h(1,x),
$$
with 
$$
\tilde{\theta}(t)=\left[\int_0^1 \frac{1}{\theta(z)}dz\right]^{-1} \int_{0}^{t}\frac{1}{\theta(z)}dz, \mbox{ and }\tilde{\theta}^{-1} \mbox{ is the inverse of } \tilde{\theta}.
$$
So, 
\begin{align*}
\mathbf{W}_m^2(\rho_{n}(0),\mathbf{S}_{\delta,h}\rho_{n}(1))&\leq \int_{0}^{1}\int_{\mathbb{R}^d}( (\tilde{\theta}^{-1})'(z))^2m(\rho_n^h(\tilde{\theta}^{-1}(z),x))|\nabla \phi_{n}^h(\tilde{\theta}^{-1}(z),x)|^2 dxdz\\
& =\int_{0}^{1}\int_{\mathbb{R}^d}(\tilde{\theta}^{-1})'(\tilde{\theta}(t)) m(\rho_n^h(t,x))|\nabla \phi_{n}^h(t,x)|^2 dxdt\\
&
=\int_0^1 \frac{1}{\theta(r)}dr \int_{0}^{1} \theta(t)\mathbf{A}_n^h(t)dt.
\end{align*}
Applying this to $\theta(t)=e^{2\lambda_{\delta}t h}$ to get  
\begin{equation}\label{Z4}
\mathbf{W}_m^2(\rho_{n}(0),\mathbf{S}_{\delta,h}\rho_{n}(1))\leq \frac{e^{-2\lambda_{\delta}h}-1}{-2\lambda_{\delta}h} \int_{0}^{1} e^{2\lambda_{\delta}t h}\mathbf{A}_n^h(t)dt.
\end{equation}
Combining \eqref{Z3} with \eqref{Z4} yield 
\begin{align*}
\frac{-\lambda_{\delta}h}{e^{-2\lambda_{\delta}h}-1}\mathbf{W}_m^2(\rho_{n}(0),\mathbf{S}_{\delta,h}\rho_{n}(1))\leq \frac{1}{2}\int_{0}^{1}\mathbf{A}_n^0(t) dt+
h(\mathbf{V}_\delta(\rho_n(0))+\mathbf{c})- \frac{1-e^{2\lambda_{\delta}h}}{-2\lambda_{\delta}}\left(\mathbf{V}_\delta(\mathbf{S}_{\delta,h}\rho_{n}(1))+\mathbf{c}\right).
\end{align*}
From \eqref{Z5}, lower semicontinuity of $\mathbf{W}^2_{m}$, continuity of $S_{\delta,h}$ and $\mathbf{V}_\delta$ under $\|\cdot\|_{L^1(\R^d)}$ , letting $n\to \infty$ we get 
\begin{align*}
\frac{-\lambda_{\delta}h}{e^{-2\lambda_{\delta}h}-1}\mathbf{W}_m^2(\mu,\mathbf{S}_{\delta,h}\xi)\leq \frac{1}{2}\mathbf{W}^2_{m}(\mu,\xi)^2+
h(\mathbf{V}_\delta(\mu)+\mathbf{c})- \frac{1-e^{2\lambda_{\delta}h}}{-2\lambda_{\delta}}\left(\mathbf{V}_\delta(\mathbf{S}_{\delta,h}\xi)+\mathbf{c}\right).
\end{align*}
So, 
\begin{align*}
\frac{-\lambda_{\delta}h}{e^{-2\lambda_{\delta}h}-1}\frac{\mathbf{W}_m^2(\mu,\mathbf{S}_{\delta,h}\xi)-\mathbf{W}_m^2(\mu,\xi)}{h}&+\frac{1}{h}\left[\frac{-\lambda_{\delta}h}{e^{-2\lambda_{\delta}h}-1}-\frac{1}{2}\right]\mathbf{W}_m^2(\mu,\xi)\\&\leq 
\mathbf{V}_\delta(\mu)+\mathbf{c}- \frac{1-e^{2\lambda_{\delta}h}}{-2\lambda_{\delta}h}\left(\mathbf{V}_\delta(\mathbf{S}_{\delta,h}\xi)+\mathbf{c}\right).
\end{align*}
Letting $h\to 0$, we obtain \eqref{Z2}. The proof is complete. 
 \end{proof}

\section{Proofs of the main results}
Now we are ready to prove our main results. 
 \begin{lemma}
\label{L-the existence of solution}
 Let $u_0\in \mathcal{P}(\mathbb{R}^d)\cap( L^2\cap \dot{H}^{-s})(\mathbb{R}^d)$. Let $m(r)=(r+\tau^{1/10})^\alpha$ or $m(r)=(r+1)^{\tau^{\frac{1}{10}}}$. Let $\tau>0$ and $\{u^k_\tau\}$ be the sequence and $\overline{u}_\tau$ be the solution of the scheme \eqref{E-scheme}- \eqref{Z12}. Then $u^k_\tau\in \dot{H}^{1-s}(\R^d)$
for every $k\geq 1$ and there holds,
 	\begin{equation}\label{Z11}
 \tau||u^k_\tau||_{\dot{H}^{1-s}}^2\leq 	\mathbf{U}(u^{k-1}_\tau)-\mathbf{U}(u^k_\tau).
 	\end{equation}
 	In particular, for any $k\in \mathbb{N}$, 
 	\begin{equation}\label{Z10}
 \mathbf{U}(u^{k}_\tau)+	\int_{0}^{k\tau}||\overline{u}_\tau(t)||_{\dot{H}^{1-s}}^2dt\leq \mathbf{U}(u^{0}_\tau)=\mathbf{U}(u_0).
 	\end{equation}
 	 Furthermore, for every sequence $\tau_n\searrow 0$ there exist a subsequence, still denoted by $\tau_n$, and $u:[0,+\infty)\to \cP(\R^d)$ such that
 	\begin{align}\label{Z15}
 	&\bar{u}_{\tau_n}\to u \mbox{ strongly in } L^2(0,T;\dot  H^\beta(\mathbb{R}^d)) \mbox{ for every } T>0,0\leq \beta<1-s,\\&
 	\bar{u}_{\tau_n}\to u \mbox{ weakly in } L^2(0,T; \dot{H}^{1-s}(\mathbb{R}^d)) \mbox{ for every } T>0.\nonumber
 	\end{align} 
 \end{lemma}
\begin{proof} Let $\mathbf{P}_t$ be the semigroup associated to the heat equation $u_t-\Delta u=0$ in $\mathbb{R}^d$.
	By the definition of $u_\tau^k$ as a minimizer of \eqref{E-scheme}: 
	$$\frac{1}{2\tau} \mathbf{W}^2_{m}(u^k_\tau,u^{k-1}_\tau)+\frac{1}{2}||u^k_\tau||_{\dot{H}^{-s}}^2\leq \frac{1}{2\tau} \mathbf{W}^2_{m}(\mathbf{P}_{h}(u^k_\tau),u^{k-1}_\tau)+\frac{1}{2}||\mathbf{P}_{h}(u^k_\tau)||_{\dot{H}^{-s}}^2,$$
	and 
	\begin{align*}
	\lim_{h\to 0}\frac{||\mathbf{P}_{h} (u^k_\tau)||_{\dot{H}^{-s}}^2-||(u^k_\tau)||_{\dot{H}^{-s}}^2}{2h}&=\langle(-\Delta) ^{-s} u^k_\tau,(\partial_h\mathbf{P}_{h} (u^k_\tau))_{h=0}\rangle 
	=-||u^k_\tau||_{\dot{H}^{1-s}}^2.
	\end{align*}
	So, to obtain \eqref{Z11} we need to show that 
	\begin{equation}\label{Z2'}
	\frac{1}{2}\limsup_{h\to 0}\frac{\mathbf{W}^2_{m}(\mathbf{P}_{h}(\xi),\mu)^2-\mathbf{W}^2_{m}(\xi,\mu)^2}{h}+ \mathbf{U}(\xi)\leq \mathbf{U}(\mu)
	\end{equation} 
	for any $\mu,\xi\in \mathcal{P}(\mathbb{R}^d)$  with $ \mathbf{V}_\delta(\mu),\mathbf{V}_\delta(\xi)<\infty, \mathbf{W}_m(\mu,\xi)<\infty.$\\
	Let $\rho_{n}$ be in Lemma \ref{apro} with $\mu^0=\mu,\mu^1=\xi$.
	For $h>0$, let $\rho^h_n(t)=\mathbf{P}_{ht}(\rho_{n}(t))\in \mathcal{P}(\mathbb{R}^d)$ and $\phi_{n}^h$ be the unique solution to 
	$\partial_t\rho_{n}^h=-\operatorname{div}(m(\rho_n^h)\nabla\phi_n^h)$
	in $[0,1]\times\mathbb{R}^d.$ \\
Applying Lemma \ref{L-Eulerian calculus} for $\varphi=0$ we get that 
		\begin{equation}\label{interchb}
	\frac{1}{2}\partial_h \int_{\mathbb{R}^d} m(\rho_n^h(t,x))|\nabla \phi_{n}^h(t,x)|^2 dx+\partial_t \mathbf{U}(\rho^h_n(t))\leq 0.
	\end{equation}
	Similar to the proof of Lemma \ref{L-flows of modified equations}, \eqref{Z2'} follows from \eqref{interchb}.   \\
	By definition \eqref{Z12}, 
	$$\int_{0}^{k\tau}||\overline{u}_\tau(t)||_{\dot{H}^{1-s}}^2dt=\sum_{j=1}^{k}\tau ||u_\tau^j||_{\dot{H}^{1-s}}^2.$$
	So, 
	 \eqref{Z10} follows from this and \eqref{Z11}. 
	\vspace{0.1cm}\\
	From \eqref{Z10}, $\{\overline{u}_{\tau_n}\}$ is bounded and hence converges weakly to $u$ in $L^2 ((0,T), \dot{H}^{1-s}(\mathbb{R}^d))$ for $T<\infty$. By compact embeddings of Sobolev spaces and the dominated convergence theorem, we deduce that $\overline{u}_{\tau_n}$ converges weakly to $u$ in $L^2 ((0,T),\dot H^{\beta}(\mathbb{R}^d))$ for $T<\infty$ and $0\leq \beta<1-s.$ 
	The proof is complete. 
 \end{proof}
  \begin{lemma} 
\label{L-solution of the main theorem}  
  Let $u_0\in \mathcal{P}(\mathbb{R}^d)\cap( L^2\cap \dot{H}^{-s})(\mathbb{R}^d)$.  
 	\label{L-weak solution} Let $\{\tau_n\}$ be a positive sequence converging to 0, $\overline{u}_\tau$ be the solution of the scheme \eqref{E-scheme}- \eqref{Z12} with $m_n(r)=(r+\tau_n^{\frac{1}{10}})^\alpha$.
 	  Then $u$ is a weak solution to  
 	\begin{equation}\label{Z17}
 	\partial_t u-\operatorname{div}(u^\alpha\nabla (-\Delta)^{-s}u )=0
 	\end{equation}
 	 with $u(t)\to u_0$ in $ \mathcal{P}(\mathbb{R}^d)\cap( L^2\cap \dot{H}^{-s})(\mathbb{R}^d)$ as $t\to 0$. 
 \end{lemma}
 \begin{proof}
 As the set of all finite sums of functions of the type $\psi(t,x)=\eta(t)\varphi(x)$ with $\eta\in C_c^\infty(0,+\infty)$ and $\varphi\in C^\infty(\R^d)$ is dense in $C_c^\infty((0,+\infty)\times \R^d)$, to prove \eqref{Z17} we only need to prove that
 \begin{align*}
-\int_0^{+\infty} \eta'(t)\langle u(t),\varphi\rangle dt= \int_{0}^{\infty}\eta(t)  \langle\dive \left(u^\alpha(t)\nabla (-\Delta) ^{-s}u(t)\right),\varphi \rangle dt, 
\end{align*}
for every $\eta\in C_c^\infty(0,+\infty)$ and $\varphi\in C^\infty(\R^d)$.

 First, we recall  $||u_0||_{\dot{H}^{-s}}^2, \mathbf{U}(u_0)<\infty$. 
Given a test function $\varphi\in C_c^\infty(\R^d)$, and $\delta,\tau>0,m(r)=(r+\tau^{\frac{1}{10}})^\alpha$. Applying Lemma \ref{L-flows of modified equations} we have 
 	\begin{equation}
 	\lambda_{\delta}\sim-\alpha\left(\frac{1}{\delta\tau^{\frac{1-\alpha}{5}}}+\alpha\tau^{\frac{\alpha-1}{10}}\right).
 	\end{equation} 
 	From Lemma \ref{L-the existence of solution} we know that $\bar{u}_{\tau_n}\to u$ strongly in $L^2(0,T,\dot{H}^\beta(\mathbb{R}^d))$ for every $T>0$ and $0\leq \beta<1-s$. Let $\delta_n=\tau_n^{1/2}$ and $\eta\in C_c^\infty(\mathbb{R}_+,\mathbb{R}_+)$. Let  $\bar{\eta}_{n}:(0,+\infty)\to (0,+\infty)$ be the simple function defined by 
 	$$\bar{\eta}_n(t):=\eta((k-1)\tau_n) \mbox{ for every } (k-1)\tau<t\leq k\tau, k\in \N.$$ 
 	By \eqref{Z12}, 
\begin{eqnarray*}
-\int_0^{+\infty} \eta'(t)\mathbf{V}_{\delta_n}(\bar{u}_{\tau_n}(t))dt&=&-\sum_{k=1}^\infty\int_{(k-1)\tau}^{k\tau}\eta'(t)\mathbf{V}_{\delta_n}(u^k_{\tau_n})dt\\
 	&=&\sum_{k=1}^\infty(\eta((k-1)\tau_n)-\eta(k\tau_n))\mathbf{V}_{\delta_n}(u^k_{\tau_n})\\
 	&=&\sum_{k=1}^\infty\eta((k-1)\tau_n)(\mathbf{V}_{\delta_n}(u_{\tau_n}^{k})-\mathbf{V}_{\delta_n}(u_{\tau_n}^{k-1})).
\end{eqnarray*}
 		
 Thanks to \eqref{Z9} in Lemma \ref{L-flows of modified equations} and $ ||u^{k-1}_{\tau_n}||_{\dot{H}^{-s}}^2-||u^k_{\tau_n}||_{\dot{H}^{-s}}^2\geq 0$, there exists $C(\varphi)>0$ such that
 \begin{align*}
 &-\int_0^{+\infty} \eta'(t)\mathbf{V}_{\delta_n}(\bar{u}_{\tau_n}(t))dt
\\&\leq \sum_{k=1}^\infty\eta((k-1)\tau_n) \tau_n\langle \dive \left((\overline{u}_{\tau_n}(t)+\tau_n^{\frac{1}{10}})^\alpha\nabla (-\Delta) ^{-s}u^k_{\tau_n}\right),\varphi \rangle\\&+C(\varphi)\tau_n\alpha\left(\frac{1}{\delta_n\tau_n^{\frac{1-\alpha}{5}}}+\tau_n^{\frac{\alpha-1}{10}}\right)\sum_{k=1}^\infty\eta((k-1)\tau_n) \left(||u^{k-1}_\tau||_{\dot{H}^{-s}}^2-||u^{k}_\tau||_{\dot{H}^{-s}}^2\right)\\&
\leq \int_{0}^{\infty}\eta_n(t) \langle \dive \left((\overline{u}_{\tau_n}(t)+\tau_n^{\frac{1}{10}})^\alpha\nabla (-\Delta) ^{-s}\overline{u}_{\tau_n}(t)\right),\varphi \rangle dt+C(\varphi)||\eta||_{L^\infty}(\tau_n^{\frac{2\alpha+3}{10}}+\tau_n^{\frac{9+\alpha}{10}})||u_0||_{\dot{H}^{-s}}^2.
 \end{align*}
 This implies that there exist $C,C(\eta,u_0)>0$ such that
  \begin{align*}
& -\int_0^{+\infty} \eta'(t)\langle \bar{u}_{\tau_n}(t),\varphi\rangle dt
 \overset{\eqref{Z14},\eqref{Z10}}\leq \int_{0}^{\infty}\eta_n(t) \langle \dive \left(\overline{u}^\alpha_{\tau_n}(t)\nabla (-\Delta) ^{-s}\overline{u}_{\tau_n}(t)\right),\varphi \rangle dt\\&+C(\varphi)||\eta||_{L^\infty}(\tau_n^{\frac{2\alpha+3}{10}}+\tau_n^{\frac{9+\alpha}{10}})||u_0||_{\dot{H}^{-s}}^2+\delta_n\int_0^{+\infty} |\eta'(t)|\mathbf{U}(u_0)dt\\&+
   \int_{0}^{\infty}\int_{\mathbb{R}^d}\eta_n(t)   |(\overline{u}_{\tau_n}(t)+\tau_n^{\frac{1}{10}})^\alpha-\overline{u}^\alpha_{\tau_n}(t)||\nabla (-\Delta) ^{-s}\overline{u}_{\tau_n}(t)||\nabla\varphi|  dx dt\\&\leq 
    \int_{0}^{\infty}\eta_n(t) \langle \dive \left(\overline{u}_{\tau_n}(t)^\alpha\nabla (-\Delta) ^{-s}\overline{u}_{\tau_n}(t)\right),\varphi \rangle dt+C(\varphi)||\eta||_{L^\infty}(\tau_n^{\frac{2\alpha+3}{10}}+\tau_n^{\frac{9+\alpha}{10}})||u_0||_{\dot{H}^{-s}}^2\\
    &+C(\eta,u_0)\tau_n^{1/2}+ C\tau_n^{\frac{1}{10}}
   \int_{0}^{\infty}\int_{\mathbb{R}^d}\eta_n(t)|\nabla (-\Delta) ^{-s}\overline{u}_{\tau_n}(t)||\nabla\varphi|  dx dt
 \end{align*}
 
By \eqref{Z15} and  letting $n\to \infty$, 
\begin{align*}
-\int_0^{+\infty} \eta'(t)\langle u(t),\varphi\rangle dt\leq \int_{0}^{\infty}\eta(t) \langle \dive \left(u^\alpha(t)\nabla (-\Delta) ^{-s}u(t)\right),\varphi \rangle dt
\end{align*}
for any $\varphi \in C^\infty_c(\mathbb{R}^d)$ and $\eta\in C^\infty_c(\mathbb{R}_+,\mathbb{R}_+)$, which gives 
\begin{align*}
-\int_0^{+\infty} \eta'(t)\langle u(t),\varphi\rangle dt= \int_{0}^{\infty}\eta(t)  \langle\dive \left(u^\alpha(t)\nabla (-\Delta) ^{-s}u(t)\right),\varphi \rangle dt,
\end{align*}
for every $\varphi \in C^\infty_c(\mathbb{R}^d)$ and $\eta\in C^\infty_c(\mathbb{R}_+)$.
Thus, we obtain \eqref{Z17}. The proof is complete. 
 \end{proof}\newline 

 \begin{proof}[Proof of Theorem  \ref{Z16}] 
 First, by interpolation inequality,  $||u_0||_{\dot{H}^{-s}}^2, \mathbf{U}(u_0)<\infty$. Given a test function $\varphi\in C_c^\infty(\R^d)$, and $\delta,\tau>0,m(r)=(r+1)^{\tau^{\frac{1}{10}}}$. Applying Lemma \ref{L-flows of modified equations} we have  
	\begin{equation}
	\lambda_{\delta}\sim-\tau^{\frac{1}{10}}\left(\frac{1}{\delta}+1\right).
	\end{equation} 
	Let $\delta_n=\tau_n^{\frac{1}{2}}$. 
	As the proof of Lemma \ref{L-weak solution}, there exists $C(\varphi)>0$ such that
	\begin{align*}
	&-\int_0^{+\infty} \eta'(t)\mathbf{V}_{\delta_n}(\bar{u}_{\tau_n}(t))dt
	\\&\leq \sum_{k=1}^\infty\eta((k-1)\tau_n) \tau_n\langle \dive \left((u^k_{\tau_n}+1)^{\tau_n^{\frac{1}{10}}}\nabla (-\Delta) ^{-s}u^k_{\tau_n}\right),\varphi \rangle\\&+C(\varphi)\tau_n^{\frac{6}{10}}\sum_{k=1}^\infty\eta((k-1)\tau_n) \left(||u^{k-1}_\tau||_{\dot{H}^{-s}}^2-||u^{k}_\tau||_{\dot{H}^{-s}}^2\right)\\&
	\leq \int_{0}^{\infty}\eta_n(t) \langle \dive \left((\overline{u}_{\tau_n}(t)+1)^{\tau_n^{\frac{1}{10}}}\nabla (-\Delta) ^{-s}\overline{u}_{\tau_n}(t)\right),\varphi \rangle dt+C(\varphi)||\eta||_{L^\infty}\tau_n^{6/10}||u_0||_{\dot{H}^{-s}}^2.
	\end{align*}
	This implies 
	\begin{align*}
	-\int_0^{+\infty} \eta'(t)\langle \bar{u}_{\tau_n}(t),&\varphi\rangle dt
	\overset{\eqref{Z14},\eqref{Z10}}\leq -\int_{0}^{\infty}\eta_n(t) \langle (-\Delta) ^{1-s}\overline{u}_{\tau_n}(t),\varphi \rangle dt\\&+C(\varphi)||\eta||_{L^\infty}\tau_n^{6/10}||u_0||_{\dot{H}^{-s}}^2+\tau_n^{\frac{1}{2}}\int_0^{+\infty} |\eta'(t)|\mathbf{U}(u_0)dt\\
	&+ \tau_n^{\frac{1}{10}}
	\int_{0}^{\infty}\int_{\mathbb{R}^d}\eta_n(t)   (\overline{u}_{\tau_n}(t)+1)^{\tau_n^{\frac{1}{10}}}\log(\overline{u}_{\tau_n}(t)+1)|\nabla (-\Delta) ^{-s}\overline{u}_{\tau_n}(t)||\nabla\varphi|  dx dt,
	\end{align*}
	here we have used the fact that 
	\begin{equation}
	0\leq (a+1)^{\tau_n^{\frac{1}{10}}}-1\leq \tau_n^{\frac{1}{10}}(a+1)^{\tau_n^{\frac{1}{10}}} \log(a+1) .
	\end{equation}
	Moreover, as $\eta\in C_c^\infty(\R_+,\R_+)$ there exists $T>0,C(\eta)>0$ such that for $n>>1$, 
		\begin{align*}
	&	\int_{0}^{\infty}\int_{\mathbb{R}^d}\eta_n(t)   (\overline{u}_{\tau_n}(t)+1)^{\tau_n^{\frac{1}{10}}}\log(\overline{u}_{\tau_n}(t)+1)|\nabla (-\Delta) ^{-s}\overline{u}_{\tau_n}(t)||\nabla\varphi|  dx dt\\ 
	&\leq C(\eta)
	\int_{0}^{T}\int_{\supp \varphi} (\overline{u}_{\tau_n}(t)+1)^{\tau_n^{\frac{1}{10}}}\log(\overline{u}_{\tau_n}(t)+1)|\nabla (-\Delta) ^{-s}\overline{u}_{\tau_n}(t)| dx dt 
		\\&\leq C
	\int_{0}^{T}\int_{\supp \varphi} (\overline{u}_{\tau_n}(t)+1)dx dt + C
	\int_{0}^{T}\int_{\mathbb{R}^d} |\nabla (-\Delta) ^{-s}\overline{u}_{\tau_n}(t)|^2 dx dt 
		\\ &\overset{\text{Holder}}\leq C(\eta,\varphi) + C(\eta,\varphi)
	\int_{0}^{T}||\overline{u}_{\tau_n}(t)||_{\dot{H}^{1-s}} dt \leq C(\eta,\varphi).
		\end{align*}
Thus, 
	\begin{align*}
-\int_0^{+\infty} \eta'(t)\langle \bar{u}_{\tau_n}(t),\varphi\rangle dt
\leq -\int_{0}^{\infty}\eta_n(t) \langle (-\Delta) ^{1-s}\overline{u}_{\tau_n}(t),\varphi \rangle dt+C(\varphi,\eta,u_0)\tau_n^{1/10}.
\end{align*}
	By \eqref{Z15} and  letting $n\to \infty$, 
	\begin{align*}
	-\int_0^{+\infty} \eta'(t)\langle u(t),\varphi\rangle dt\leq -\int_{0}^{\infty}\eta(t) \langle  (-\Delta) ^{1-s}u(t),\varphi \rangle dt
	\end{align*}
	for any $\varphi \in C^\infty_c(\mathbb{R}^d)$ and $\eta\in C^\infty_c(\mathbb{R}_+,\mathbb{R}_+)$, which gives 
	\begin{align*}
	\int_0^{+\infty} \eta'(t)\langle u(t),\varphi\rangle dt= \int_{0}^{\infty}\eta(t)  \langle (-\Delta) ^{1-s}u(t),\varphi \rangle dt,
	\end{align*}
	for every $\varphi \in C^\infty_c(\mathbb{R}^d)$ and $\eta\in C^\infty_c(\mathbb{R}_+)$.
	Thus, we obtain \eqref{Z18}. The proof is complete. 
\end{proof}\newline

Let $p\in (1,+\infty)$ and let $f_p:[0,+\infty)\to \R$ be the function defined by $f_p(r):=\frac{r^p}{p-1}$. Let $\tau>0$ and $m(r)=(r+\tau^{\frac{1}{10}})^\alpha$, we consider the function $P_p:[0,+\infty)\to \R$ defined by
$$P_p(r):=\int_0^rm(z)f_p''(z)dz.$$
Let $\cU:P_2(\R^d)\to (-\infty,+\infty]$ be the map defined by
$$\cU(u):=\int_{\R^d}f_p(u(x))dx.$$
\begin{lemma}
	\label{L-decay of scheme}
	Let $u_0\in \mathcal{P}_2(\R^d)\cap (L^2\cap \dot{H}^{-s})(\R^d)$ such that $\|u_0\|^2_{\dot{H}^{-s}(\R^d)}<+\infty$. Let $\tau>0$ and $\{u^k_\tau\}$ be the sequence and $\overline{u}_\tau$ be the solution of the scheme \eqref{E-scheme}- \eqref{Z12} with $m(r)=(r+\tau^{\frac{1}{10}})^\alpha$. Let $g\in C^{2}([0,\infty),\mathbb{R}_+)$ be convex such that $g(0)=g'(0)=g''(0)=0$. Then, 
	\begin{align}
	\label{F-Decay of entropies}
	0\leq \tau||\cG(u^k_\tau)||_{ \dot{H}^{1-s}(\mathbb{R}^d)}^2\leq  \int_{\mathbb{R}^d} g(u^{k-1}_\tau(x)) dx-\int_{\mathbb{R}^d} g(u^{k}_\tau(x)) dx, \mbox{ for every } k\geq 1.
	\end{align}
	with $\cG(r)=\int_{0}^{r}\sqrt{m(z)g^{''}(z) } dz.$ In particular, 
		\begin{equation}\label{F-Decay of entropies'}
	C(d,s)\tau||\cG(u^k_\tau)||_{L^{\frac{2d}{d-2(1-s)}}}^2\leq  ||u^{k-1}_\tau||_{L^p(\mathbb{R}^d)}^p-||u^{k}_\tau||_{L^p(\mathbb{R}^d)}^p.
	\end{equation}	
\end{lemma}
\begin{proof}
	For $\delta>0$,  we consider the map $\mathbf{U}_\delta:P_2(\R^d)\to (-\infty,+\infty]$ defined by
	$$\mathbf{U}_\delta(u):=\int g(u(x)) dx+\delta \mathbf{U}(u).$$
	We consider  
	\begin{equation}
	\label{E-solutions of approximations'}
	\partial_t v_t-\Delta \mathbf{G}(v_t)-\delta \Delta v_t=0 \mbox{ in } (0,+\infty)\times \mathbb{R}^d.
	\end{equation}
	where $$\mathbf{G}(r)=\int_0^rm(z) g''(z) dz.$$
	Clearly, \eqref{E-solutions of approximations'} has the uniqueness solution with initial data $v_0\in \mathcal{P}(\mathbb{R}^d)$. \\
	We define the semigroup $\mathbf{K}_{\delta}$ by $\mathbf{K}_{\delta,t}v_0=v_t$ for every $t>0$.\\
	By the definition of $u_\tau^k$ as a minimizer of \eqref{E-scheme}: 
	$$\frac{1}{2\tau} \mathbf{W}^2_{m}(u^k_\tau,u^{k-1}_\tau)+\frac{1}{2}||u^k_\tau||_{\dot{H}^{-s}}^2\leq \frac{1}{2\tau} \mathbf{W}^2_{m}(\mathbf{K}_{\delta,h}(u^k_\tau),u^{k-1}_\tau)+\frac{1}{2}||\mathbf{K}_{\delta,h}(u^k_\tau)||_{\dot{H}^{-s}}^2,$$
	and 
	\begin{align*}
	\lim_{h\to 0}\frac{||\mathbf{K}_{\delta,h} (u^k_\tau)||_{\dot{H}^{-s}}^2-||u^k_\tau||_{\dot{H}^{-s}}^2}{2h}&=\langle(-\Delta) ^{-s} u^k_\tau,(\partial_h\mathbf{K}_{\delta,h} (u^k_\tau))_{h=0}\rangle \\&= \langle(-\Delta) ^{-s} u^k_\tau,\Delta \mathbf{G}(u^k_\tau)+\delta \Delta u^k_\tau\rangle\\&
	=-\delta||u^k_\tau||_{\dot{H}^{1-s}}^2-\langle u^k_\tau,(-\Delta)^{1-s} \mathbf{G}(u^k_\tau)\rangle.
	\end{align*}
	So, we get that
	\begin{align}
	\label{T-temp}
	0\leq \tau\langle u^k_\tau,(-\Delta)^{1-s} \mathbf{G}(u^k_\tau)\rangle\leq \limsup_{h\to 0}\frac{\mathbf{W}^2_{m}(\mathbf{K}_{\delta,h}(u^k_\tau),u^{k-1}_\tau)-\mathbf{W}^2_{m}(u^k_\tau,u^{k-1}_\tau)}{2h}.
	\end{align}
	We prove that 
	\begin{equation}\label{T-temp'}
	\limsup_{h\to 0}\frac{\mathbf{W}^2_{m}(\mathbf{K}_{\delta,h}(\xi),\mu)-\mathbf{W}^2_{m}(\xi,\mu)}{2h}+\mathbf{U}_\delta(\xi)\leq \mathbf{U}_\delta(\mu)
	\end{equation}
	for any $\mu,\xi\in \mathcal{P}(\mathbb{R}^d)$  with $ \mathbf{U}_\delta(\mu),\mathbf{U}_\delta(\xi)<\infty, \mathbf{W}_m(\mu,\xi)<\infty.$
	\vspace{0.1cm}\\
	Indeed, let $\rho_{n}$ be in Lemma \ref{apro} with $\mu^0=\mu,\mu^1=\xi$.
	Set for $h>0, t>0$, $\rho^h_n(t)=\mathbf{K}_{\delta,ht}\rho_{n}(t) \in \mathcal{P}(\mathbb{R}^d).$
	Let $\phi_{n}^h$ be a unique solution to 
	$\partial_t\rho_{n}^h(t,x)=-\operatorname{div}(m(\rho_n^h(t,x))\nabla\phi_n^h(t,x))$
	in $[0,1]\times\mathbb{R}^d.$ \\
	We prove that  
	\begin{equation}\label{interchc}
	\frac{1}{2}\partial_h \int_{\mathbb{R}^d} m(\rho_n^h(t,x))|\nabla \phi_{n}^h(t,x)|^2 dx +\partial_t \mathbf{U}_\delta(\rho^h_n(t))\leq 0 
	\end{equation}
	for any $t\in [0,1]$ and $h\geq 0$. \newline
	Indeed, \\
	\textit{Step 1:} one has 
	\begin{align*}
	&\partial_t \mathbf{U}_\delta(\rho^h_n(t))=\langle g'(\rho^h_n(t)), \partial_t\rho^h_n(t)\rangle+\delta \langle \mathbf{U}'(\rho^h_n(t)), \partial_t\rho^h_n(t)\rangle\\&~~~=- \langle g'(\rho^h_n(t)), \operatorname{div}(m(\rho_n^h(t))\nabla\phi_n^h(t))\rangle-\delta \langle \mathbf{U}'(\rho^h_n(t)), \operatorname{div}(m(\rho_n^h(t))\nabla\phi_n^h(t)\rangle
	\\&~~~=\int_{\mathbb{R}^d}  m(\rho_n^h(t))g^{''}(\rho^h_n(t))\nabla  \rho_n^h(t)\nabla \phi_n^h(t)+\delta\int_{\mathbb{R}^d} \mathbf{U}^{''}(\rho^h_n(t)) m(\rho_n^h(t))\nabla\phi_n^h(t)\nabla \rho_n^h(t).
	\end{align*}
	Since $\mathbf{G}'(\rho^h_n(t))=m(\rho^h_n(t)) g^{''}(\rho^h_n(t)), \mathbf{U}^{''}(\rho^h_n(t)) m(\rho_n^h(t))=1$, 
	\begin{equation}\label{Z8'}
	\partial_t \mathbf{U}_\delta(\rho^h_n(t))=- \int_{\mathbb{R}^d}   \mathbf{G}(\rho^h_n(t)) \Delta\phi_n^h(t)-\delta\int_{\mathbb{R}^d} \rho_n^h(t)\Delta\phi_n^h(t).
	\end{equation}
	\textit{Step 2:} By \eqref{Z1}, 
	\begin{align*}
	\frac{1}{2} \partial_h \int_{\mathbb{R}^d} m(\rho_n^h(t,x))|\nabla \phi_{n}^h(t,x)|^2 dx =-\frac{1}{2}\int_{\mathbb{R}^d} \partial_h\rho_{n}^h m'(\rho_{n}^h) |\nabla \phi_{n}^h|^2+\int_{\mathbb{R}^d} (\partial_t\partial_h\rho_{n}^h) \phi_{n}^h.
	\end{align*}
	Note that 
	\begin{equation}\label{Z6b}
	\partial_h \rho^h_n(t)=t\Delta \mathbf{G}(\rho^h_n(t))+\delta t \Delta \rho^h_n(t).
	\end{equation}
	Hence 
	\begin{align}
	\partial_t\partial_h \rho^h_n(t)=
	\delta \Delta \rho^h_n(t)+\Delta \mathbf{G}(\rho^h_n(t))-t\Delta \left[ \operatorname{div}(m(\rho_n^h(t))\nabla\phi_n^h(t))\left(\mathbf{G}'( \rho^h_n(t))+\delta\right) \right].\label{Z7b}
	\end{align}
	By \eqref{Z6b}  and $m''\leq 0,\mathbf{G}'\geq 0$, we get
	\begin{align*}
	&-\frac{1}{2}\int_{\mathbb{R}^d} \partial_h\rho_{n}^h m'(\rho_{n}^h) |\nabla \phi_{n}^h|^2=-\frac{t}{2}\int_{\mathbb{R}^d}  \Delta\left(  \mathbf{G}(\rho^h_n(t))+\delta \rho^h_n(t)\right) m'(\rho_{n}^h) |\nabla \phi_{n}^h|^2\\&=\frac{t}{2} \int_{\mathbb{R}^d}\nabla \left(  \mathbf{G}(\rho^h_n(t))+\delta \rho^h_n(t)\right)\nabla\left( m'(\rho_{n}^h) |\nabla \phi_{n}^h|^2\right)\\&=\frac{t}{2}\int_{\mathbb{R}^d} \left(\mathbf{G}'(\rho^h_n(t))+\delta\right) m''(\rho_{n}^h) | \nabla \rho^h_n(t)|^2 |\nabla \phi_{n}^h|^2+\frac{t}{2}\int_{\mathbb{R}^d}  \left(\mathbf{G}'(\rho^h_n(t))+\delta\right) m'(\rho_{n}^h) \nabla \rho^h_n(t)\nabla (|\nabla \phi_{n}^h|^2)
	\\&\leq -t\int_{\mathbb{R}^d} \left(\tilde{\mathbf{G}}(\rho_{n}^h)+ \delta m(\rho_{n}^h) \right)\Delta (\frac{1}{2}|\nabla \phi_{n}^h|^2),
	\end{align*}
	where $
	\tilde{\mathbf{G}}(r)=\int_{0}^{r} \mathbf{G}'(a)m'(a)da$.\\
	By \eqref{Z7b},
	\begin{align*}
	&\int_{\mathbb{R}^d} (\partial_t\partial_h\rho_{n}^h) \phi_{n}^h=
	\int_{\mathbb{R}^d} \left(\mathbf{G}(\rho^h_n(t))+\delta \rho^h_n(t) \right)\Delta \phi_{n}^h
	- t \int_{\mathbb{R}^d} \operatorname{div}(m(\rho_n^h(t))\nabla\phi_n^h(t))\left(\mathbf{G}'( \rho^h_n(t))+\delta\right)  \Delta \phi_{n}^h\\
	&
	\overset{\eqref{Z8'}}= -\partial_t \mathbf{U}_\delta(\rho^h_n(t))+ t \int_{\mathbb{R}^d} \left(\mathbf{G}'( \rho^h_n(t))+\delta\right) m(\rho_n^h(t))\nabla\phi_n^h(t) \nabla(\Delta \phi_{n}^h)\\
	&+ t \int_{\mathbb{R}^d}  \mathbf{G}''( \rho^h_n(t))m(\rho_n^h(t))\left(\nabla\phi_n^h(t) \nabla\rho_{n}^h(t) \right)\Delta \phi_{n}^h
	\\&
	= -\partial_t \mathbf{U}_\delta(\rho^h_n(t))+ t \int_{\mathbb{R}^d} \left(\mathbf{G}'( \rho^h_n(t))+\delta\right) m(\rho_n^h(t))\nabla\phi_n^h(t) \nabla(\Delta \phi_{n}^h)+ t \int_{\mathbb{R}^d} \left(\nabla\phi_n^h(t) \nabla\hat{\mathbf{G}}(\rho_{n}^h(t)) \right)\Delta \phi_{n}^h,
	\end{align*} 
	where $
	\hat{\mathbf{G}}(r)=\int_{0}^{r} \mathbf{G}''(a)m(a)da\geq 0$.\\
	Since $
	\tilde{\mathbf{G}}(r)=\mathbf{G}'(r)m(r)-\hat{\mathbf{G}}(r)$, 
	\begin{align*}
	&\int_{\mathbb{R}^d} (\partial_t\partial_h\rho_{n}^h) \phi_{n}^h
	= -\partial_t \mathbf{U}_\delta(\rho^h_n(t))+ t \int_{\mathbb{R}^d} \left(\mathbf{G}'( \rho^h_n(t)) m(\rho_n^h(t))+\delta m(\rho_n^h(t))\right) \nabla\phi_n^h(t) \nabla(\Delta \phi_{n}^h)\\&- t \int_{\mathbb{R}^d} \hat{\mathbf{G}}(\rho_{n}^h(t))(\Delta \phi_{n}^h)^2- t \int_{\mathbb{R}^d} \hat{\mathbf{G}}(\rho_{n}^h(t))\nabla\phi_n^h(t) \nabla(\Delta \phi_{n}^h)\\&\leq 
	-\partial_t \mathbf{U}_\delta(\rho^h_n(t))+ t \int_{\mathbb{R}^d} \left(\tilde{\mathbf{G}}( \rho^h_n(t)) +\delta m(\rho_n^h(t))\right) \nabla\phi_n^h(t) \nabla(\Delta \phi_{n}^h).
	\end{align*} 
	Therefore, 
	\begin{align*}
	&\frac{1}{2} \partial_h \int_{\mathbb{R}^d} m(\rho_n^h(t,x))|\nabla \phi_{n}^h(t,x)|^2 dx +\partial_t \mathbf{U}_\delta(\rho^h_n(t))\\&\leq   t \int_{\mathbb{R}^d} \left(\tilde{\mathbf{G}}( \rho^h_n(t))+\delta m(\rho_n^h(t))\right) \left(\nabla\phi_n^h(t) \nabla\Delta \phi_{n}^h-\Delta (\frac{1}{2}|\nabla \phi_{n}^h|^2)\right).	 
	\end{align*}
	Since
	\begin{align*}
	&-\Delta (\frac{1}{2}|\nabla \phi_{n}^h|^2)+\nabla\phi_n^h(t) \nabla\Delta \phi_{n}^h=-|D^2 \phi_{n}^h|^2\leq 0,
	\end{align*}
	we get \eqref{interchc}. 
\vspace{0.1cm}\\
	As the proof of step 2 of Lemma \ref{L-flows of modified equations}, \eqref{T-temp'} follows from \eqref{interchc}.   \\
	Combining  \eqref{T-temp} and \eqref{T-temp'} we obtain 
	$$0\leq \tau\bigg\langle u^k_\tau,(-\Delta)^{1-s}(\mathbf{G}(u^k_\tau))\bigg\rangle\leq  \mathbf{U}_\delta(u^{k-1}_\tau)-\mathbf{U}_\delta(u^k_\tau).$$ 
	Letting $\delta\to 0$ we get
	\begin{equation}
	\label{Z19} 0\leq \tau\bigg\langle u^k_\tau,(-\Delta)^{1-s}(\mathbf{G}(u^k_\tau))\bigg\rangle\leq  \int_{\mathbb{R}^d} g(u^{k-1}_\tau(x)) dx-\int_{\mathbb{R}^d} g(u^{k}_\tau(x)) dx.
	\end{equation}
	Since $\cG'(t)=\sqrt{\mathbf{G}'(t)}$, one has
	\begin{align*}
	( \cG(a)-\cG(b))^2=\bigg(\int_a^b\cG'(t)dt\bigg)^2\leq \bigg(\int_a^b1dt\bigg)\bigg(\int_a^b (\cG'(t))^2dt\bigg)=(b-a)(G(b)-G(a)).
	\end{align*}
	for any $b\geq a$.
	So, 
	\begin{equation}\label{Z20}
	||\cG (u^k_\tau)||_{\dot{H}^{1-s}(\mathbb{R}^d)}^2= \bigg\langle \cG (u^k_\tau),(-\Delta)^{1-s}(\cG (u^k_\tau))\bigg\rangle\leq  \bigg\langle u^k_\tau,(-\Delta)^{1-s}(\mathbf{G}(u^k_\tau))\bigg\rangle.
	\end{equation}
	Thus, \eqref{F-Decay of entropies} follows from \eqref{Z19} and \eqref{Z20}. The proof is complete. 
\end{proof}
\begin{lemma}
\label{L-p norm decay}
	Let $u_0\in  \cP_2(\R^d)\cap \dot{H}^{-s}(\R^d)$ and let $\tau>0$, $\{u^k_\tau\}$ be the sequence and $\overline{u}_\tau$ be the solution of the scheme \eqref{E-scheme}-\eqref{Z12} with $m(r)=(r+\tau^{\frac{1}{10}})^\alpha$. Let $d\geq 2$, $0<s<1$ and $C(d,s)$ as in Lemma \ref{L-decay of scheme}. Then for any $L\in \N$, $1\leq p\leq 2^L$, we have 
		\begin{equation}\label{Z30}
		||\overline{u}_\tau(t)||_{L^p(\mathbb{R}^d)}\leq (\lambda ( t+ k_0\tau))^{-\frac{(1-1/p)d}{d\alpha+2(1-s)}},
		\end{equation}
		where  
		\begin{equation}
		\lambda=\frac{C(d,s)}{\sup_{1\leq n\leq L}\sup _{k\geq k_0}k \left(\left(1+\frac{1}{k-1}\right)^{\frac{(2^n-1)d}{d\alpha+2(1-s)}}-1\right)}
		\end{equation} 
		and 	\begin{equation}
		(\lambda \tau)^{-\frac{d}{d\alpha+2(1-s)}}\geq ||u_0||_{\infty}.
		\end{equation}
\end{lemma}
\begin{proof} Applying \eqref{F-Decay of entropies'}  to $g(z)=z^p$ for $p>2$ to get 
	\begin{equation*}
	C(d,s)\tau||\cG(u^k_\tau)||_{\frac{2d}{d-2(1-s)}}^2\leq  ||u^{k-1}_\tau||_{L^p(\mathbb{R}^d)}^p-||u^{k}_\tau||_{L^p(\mathbb{R}^d)}^p.
	\end{equation*}	
	From now on, for simplicity we write $||f||_q$ for $||f||_{L^q(\mathbb{R}^d)}$.\\
Since $\cG(r)=\int_{0}^{r}\sqrt{m(z)g''(z) } dz\geq \frac{2\sqrt{p(p-1)}}{\alpha+p} r^{\frac{\alpha+p}{2}}$, there exists $c>0$ such that
	\begin{equation}
c\tau||u^k_\tau||_{\frac{(\alpha+p)d}{d-2(1-s)}}^{\alpha+p}\leq  ||u^{k-1}_\tau||_{p}^p-||u^{k}_\tau||_{p}^p.
\end{equation}	
	By interpolation inequality, one has for $1\leq q<p$
	\begin{equation*}
	c\tau ||u^k_\tau||_{p}^{\frac{\frac{\alpha+p}{q}-1+\frac{2(1-s)}{d}}{\frac{1}{q}-\frac{1}{p}}}||u^k_\tau||_{q}^{-\frac{\frac{\alpha}{p}+\frac{2(1-s)}{d}}{\frac{1}{q}-\frac{1}{p}}}\leq  ||u^{k-1}_\tau||_{p}^p-||u^k_\tau||_{p}^p.
	\end{equation*}
	Now we apply this to $p=2^{n}$ and $q=2^{n-1}$ for $n=1,...$, 
	\begin{equation*}
	c\tau	||u^k_\tau||_{2^n}^{2\alpha+2^{n}+\frac{2^{n+1}(1-s)}{d}}||u^k_\tau||_{2^{n-1}}^{-\alpha-\frac{2^{n+1}(1-s)}{d}}\leq  ||u^{k-1}_\tau||_{2^n}^{2^n}-||u^k_\tau||_{2^n}^{2^n}.
	\end{equation*}
		For every $\lambda>0, n,k\in \N$, we consider 
	\begin{equation}
	\mathbf{g}_{2^n,k}=	(\lambda ( k+k_0)\tau)^{-\frac{(1-2^{-n})d}{d\alpha+2(1-s)}}~\text{for} ~k_0\geq 1.
	\end{equation}
	One has for $k\geq 1$ and $n\geq 1$, 
	\begin{align*}
	\frac{\mathbf{g}_{2^n,k-1}^{2^n}-\mathbf{g}_{2^n,k}^{2^n}}{ \mathbf{g}_{2^n,k}^{2\alpha+2^{n}+\frac{2^{n+1}(1-s)}{d}}\mathbf{g}_{2^{n-1},k}^{-\alpha-\frac{2^{n+1}(1-s)}{d}}}&=\lambda\tau \frac{ ( k-1+k_0)^{-\frac{(2^n-1)d}{d\alpha+2(1-s)}}-	( k+k_0)^{-\frac{(2^n-1)d}{d\alpha+2(1-s)}}}{(k+k_0)^{-1-\frac{(2^n-1)d}{d\alpha+2(1-s)}}}\\&= \lambda \tau(k+k_0) \left(\left(1+\frac{1}{k-1+k_0}\right)^{\frac{(2^n-1)d}{d\alpha+2(1-s)}}-1\right).
	\end{align*}
	If \begin{equation}
	\lambda=\lambda(k_0,L) =\frac{c}{\sup_{1\leq n\leq L}\sup _{k\geq k_0}k \left(\left(1+\frac{1}{k-1}\right)^{\frac{(2^n-1)d}{d\alpha+2(1-s)}}-1\right)}
	\end{equation}
	and 
	\begin{equation}
	(\lambda k_0 \tau)^{-\frac{d}{d\alpha+2(1-s)}}\geq ||u_0||_{\infty},
	\end{equation}
	then, for any $n=1,...,L$ 
	\begin{equation}
	c\tau \mathbf{g}_{2^n,k}^{2\alpha+2^{n}+\frac{2^{n+1}(1-s)}{d}}\mathbf{g}_{2^{n-1},k}^{-\alpha-\frac{2^{n+1}(1-s)}{d}}\geq \mathbf{g}_{2^n,k-1}^{2^n}-\mathbf{g}_{2^n,k}^{2^n}, 
	\end{equation}
	and $~\mathbf{g}_{2^n,0}\geq ||\overline{u}_\tau(0)||_{2^n},\mathbf{g}_{1,k}=1$.\\
	Therefore, using induction it is easy to check that for any $n=1,...,L$ 
	\begin{equation}
	||u^k_\tau||_{2^n}\leq \mathbf{g}_{2^n,k}=(\lambda ( k+k_0)\tau)^{-\frac{(1-2^{-n})d}{d\alpha+2(1-s)}}.
	\end{equation}
	This gives
	\begin{equation}
	||\overline{u}_\tau(t)||_{2^n}\leq (\lambda  ( t+ k_0\tau))^{-\frac{(1-2^{-n})d}{d\alpha+2(1-s)}}
	\end{equation}
	for any $n=1,...,L$. \\
	By interpolation inequality, we obtain 
	\begin{equation}
	||\overline{u}_\tau(t)||_{L^p(\mathbb{R}^d)}\leq (\lambda   ( t+ k_0\tau))^{-\frac{(1-1/p)d}{d\alpha+2(1-s)}}
	\end{equation}
	for any $1\leq p\leq 2^L$. So, we finish to prove  \eqref{Z30}. The proof is complete. 
\end{proof}
\begin{lemma}
	\label{L-p norm decayb}
	 Let $d\geq 2$, $0<s<1$ and $C(d,s)$ as in Lemma \ref{L-decay of scheme}. Let $u_0\in  \cP_2(\R^d)\cap \dot{H}^{-s}(\R^d)$ and let $\tau>0$, $\{u^k_\tau\}$ be the sequence and $\overline{u}_\tau$ be the solution of the scheme \eqref{E-scheme}-\eqref{Z12} with $m(r)=(r+\tau^{\frac{1}{10}})^\alpha$. Then there exists $C>0$ such that for any $t_1>0$ and $\Lambda>0$, we have
		\begin{equation}\label{Z21}
	\sup_{t\geq t_1}||(\overline{u}_\tau(t)-\Lambda)_+||_{2}^2+\int_{t_1+\tau}^{\infty}||(\overline{u}_\tau(t)-\Lambda)_+||_{\frac{(\alpha+2)d}{d-2(1-s)}}^{\alpha+2} dt\leq  C ||(\overline{u}_\tau(t_1)-\Lambda)_+||_{L^2}^2.
	\end{equation}
	Moreover, if $\tau<2^{-4}$ and $u_0\in L^2$, there exists $\Lambda_0=\Lambda(||u_0||_{L^2},s,\alpha,d)>0$ such that 
		\begin{equation}\label{Z25}
	\sup_{t\geq 1}||(\overline{u}_\tau(t)-\Lambda_0)_+||_{2}^2+\int_{1}^{\infty}||(\overline{u}_\tau(t)-\Lambda)_+||_{\frac{(\alpha+2)d}{d-2(1-s)}}^{\alpha+2} dt \leq C  \tau^{\frac{d(\alpha+1)}{2(1-s)}+2}.
	\end{equation}
\end{lemma}

\begin{proof}  Let $\Lambda>0$.  Applying \eqref{F-Decay of entropies'}  to $g(z)=(z-\Lambda)_+^p$ for $p>2$ to get 
	\begin{equation*}
	C(d,s)\tau||\cG(u^k_\tau)||_{\frac{2d}{d-2(1-s)}}^2\leq  ||(u^{k-1}_\tau-\Lambda)_+||_{p}^p-||(u^{k}_\tau-\Lambda)_+||_{p}^p.
	\end{equation*}	
	Letting $p\to 2$, there exists $c>0$ such that 
	\begin{equation}
	c\tau||(u^k_\tau-\Lambda)_+||_{\frac{(\alpha+2)d}{d-2(1-s)}}^{\alpha+2}\leq  ||(u^{k-1}_\tau-\Lambda)_+|_{L^2}^2-||(u^k_\tau-\Lambda)_+||_{2}^2.
	\end{equation}	
	for any $\Lambda>0$. 
	So, 
	\begin{equation}
	\sup_{t\geq t_1}||(\overline{u}_\tau(t)-\Lambda)_+||_{2}^2+c\int_{t_1+\tau}^{\infty}||(\overline{u}_\tau(t)-\Lambda)_+||_{\frac{(\alpha+2)d}{d-2(1-s)}}^{\alpha+2} dt\leq  3 ||(\overline{u}_\tau(t_1)-\Lambda)_+||_{L^2}^2.
	\end{equation}
	for any $t_1\geq 0$ and $\Lambda>0$. This implies \eqref{Z21}. \\
	Next, let us set $T_n=1-2^{-n}$, $\Lambda_n=\Lambda(1-2^{-n})$, and $u_{\tau,n}(t)=(\overline{u}_\tau(t)-\Lambda_n)_+$, for every $1\leq n\leq \left[\frac{|\log(\tau)|}{\log(2)}\right]-3$. 
	\\
	And, we denote a level set of energy by
	\begin{equation*}
	U_n=\sup_{t\geq T_n}||u_{\tau,n}(t)||_{2}^2+\int_{T_n}^{\infty}||u_{\tau,n}(t)||_{\frac{(\alpha+2)d}{d-2(1-s)}}^{\alpha+2} dt
	\end{equation*}
	We have  
	\begin{equation}\label{Z22}
	U_1\leq C||u(0)||_{L^2}^2
	\end{equation}
and 
	\begin{equation*}
U_n\leq C\inf_{t\in [T_{n-1},T_{n}]} ||u_{\tau,n}(t)||_{2}^2\leq C2^{n} ||u_{\tau,n}(t)||_{L^2([T_{n-1},\infty)\times\mathbb{R}^d)}^2
\end{equation*}
	Since
	\[\left\{	\begin{array}{cl}
	& u_{\tau,n}\leq u_{\tau,n-1}, 
	\\
	& \mathbf{1}_{u_{\tau,n}>0}\leq \mathbf{1}_{u_{\tau,n-1}>\Lambda 2^{-n}},
	\\
	& \|u_{\tau,n}\|^{\frac{d(\alpha+2)+4(1-s)}{d}}_{L^{\frac{d(\alpha+2)+4(1-s)}{d}}([T_{n},\infty)\times\mathbb{R}^N)}\leq C U_{n}^{\frac{d+2(1-s)}{d}} ,
	\end{array}\right.
	\]
	one has for any $n\geq 1$
	\begin{align*}
	U_n  \leq  C2^{n} ||u_{\tau,n-1}(t)||_{L^2([T_{n-1},\infty)\times\mathbb{R}^d)}^2
\leq  C2^{n} \left(\frac{2^n}{\Lambda}\right)^{\frac{d\alpha+4(1-s)}{d}}  U_{n-1}^{\frac{d+2(1-s)}{d}}
	\end{align*}
	Combining this with \eqref{Z22}, we can find  $\Lambda_0=\Lambda(||u_0||_{L^2},d,s,\alpha)>0$ such that 
	\begin{equation}
	U_n\leq 2^{-c_0n},~~c_0=\frac{d(\alpha+1)}{2(1-s)}+2
	\end{equation}
	for any $n=2,...,\left[\frac{|\log(\tau)|}{\log(2)}\right]-3.$ \\Therefore, 
	\begin{align*}
	\sup_{t\geq 1}||(\overline{u}_\tau(t)-\Lambda_0)_+||_{2}^2+\int_{1}^{\infty}||(\overline{u}_\tau(t)-\Lambda_0)_+||_{\frac{(\alpha+2)d}{d-2(1-s)}}^{\alpha+2} dt\leq  2^{-c_0\left(\left[\frac{|\log(\tau)|}{\log(2)}\right]-3\right)}\leq C \tau^{c_0}
	\end{align*}which implies \eqref{Z25}. The proof is complete. 
\end{proof}
\vspace{0.1cm}\\
 \begin{proof}[Proof of Theorem \ref{T-solutions of fractional equations}]
 		i), ii), iii)  and iv) follow from Lemmas \ref{L-the existence of solution}, \ref{L-solution of the main theorem} and \ref{L-p norm decay}, \ref{L-p norm decayb}. 
 	

 \end{proof}


\end{document}